\documentclass[11pt]{article}
\usepackage{amsmath,amssymb,algorithm,algpseudocode,graphicx}
\usepackage{hyperref}
\usepackage{amsmath}
\usepackage{amsthm}
\numberwithin{equation}{section}
\newtheorem{theorem}{Theorem}[section]
\newtheorem{corollary}{Corollary}[section]
\newtheorem{proposition}{Proposition}[section]
\newtheorem{assumption}{Assumption}
\newtheorem{lemma}{Lemma}[section]

\newcommand{\argmax}{\operatorname{argmax}}
\usepackage{geometry}
\usepackage{booktabs}
\usepackage{siunitx}
\usepackage{subcaption}
\usepackage{multirow}

\begin{document}

\title{A Tree-Based Localized Reduced Basis Method for Parametrized Parabolic PDEs
\thanks{ \textbf{Funding: }This research is partially supported  by the NSERC Discovery Grant RGPIN-2021-04311}}

\author{Mohamed Barakat\thanks{Department of Mathematics and Statistics, University of Ottawa, Ottawa, ON, K1N 6N5}
\ and Diane Guignard\footnotemark[2]}
\date{}

\maketitle
\begin{abstract}
    We propose a localized reduced basis method for parametrized parabolic partial differential equations based on an adaptive partitioning of the parameter domain using hierarchical binary trees. A POD-Greedy procedure is employed to construct reduced spaces associated with each parameter subdomain. A certified \emph{a posteriori} error estimator is derived from an inf-sup stability analysis and shown to be equivalent to the underlying approximation error. This equivalence, together with suitable regularity assumptions on the parameter-to-solution map, enables a rigorous convergence analysis and is used to derive a strategy for selecting the number of POD modes to retain in each local reduced space. Numerical experiments for a parametrized convection--diffusion problem demonstrate the effectiveness of the proposed approach and confirm the predicted convergence behavior.
\end{abstract}

\section{Introduction}
Reduced order modeling has become a central tool in the numerical approximation of parametriz-ed partial differential equations (PDEs). Among the various approaches, the reduced basis (RB) method, introduced in \cite{almroth1978automatic,noor1980reduced} and further analyzed in \cite{porsching1985estimation,rheinboldt1992theory}, has emerged as one of the most successful techniques for constructing efficient low-dimensional surrogates of parametrized models. RB methods are particularly well suited for applications requiring repeated numerical simulations, such as optimization, uncertainty quantification, inverse problems, parameter estimation and real-time control, where the governing PDE must be solved many times for different parameter values under limited computational resources \cite{boyaval2008reduced,boyaval2009reduced,nguyen2010reduced}. Comprehensive treatments of the methodology, together with its theoretical foundations and computational aspects, can be found in \cite{quarteroni2011certified,hesthaven2016certified}.

In general, the RB methodology relies on an offline-online decomposition. During the offline stage, a low-dimensional approximation space is constructed from a collection of high-fidelity solutions, referred to as snapshots, computed for carefully selected parameter values. These snapshots capture the dominant features of the solution manifold and form the basis of an efficient reduced-order model. During the online stage, the reduced solution is obtained by projecting the governing equations onto the reduced space, yielding a system whose dimension is independent of the underlying high-fidelity discretization \cite{rozza2008reduced}.

For time-dependent problems, the construction of efficient reduced spaces is more challenging since the solution depends simultaneously on the parameter and the temporal variable. Consequently, the reduced basis must accurately represent both the parameter dependence and the temporal evolution of the solution. Several approaches have been proposed to address this issue. A straightforward strategy consists of treating every time instance as an independent snapshot and applying a greedy algorithm directly to the resulting space-time solution manifold. Although conceptually simple, this approach may stagnate and fail to achieve a high level of approximation efficiency \cite{grepl2005posteriori}. A more efficient and widely adopted strategy is the Proper Orthogonal Decomposition (POD)-Greedy algorithm, in which POD is first employed to compress the temporal trajectories associated with a selected parameter value before a greedy procedure determines the next parameter according to a certified error estimator \cite{haasdonk2008reduced, hesthaven2016certified}. An important ingredient of the development of certified reduced basis methods for parabolic PDEs is an efficient error estimator. Residual-based {a posteriori} error estimators for parametrized parabolic problems have been developed in
\cite{grepl2005posteriori,knezevic2011reduced,nguyen2009reduced}.
These estimators provide rigorous upper bounds for the $L^2$ error at individual time steps. However, they do not explicitly account for the contribution of the discrete time derivative error, and no corresponding lower bound establishing the efficiency of the estimator with respect to this norm is provided.

In many applications, the parameter-to-solution map is sufficiently smooth that a single global reduced space provides an accurate approximation of the entire solution manifold, often yielding exponential convergence with respect to the reduced basis dimension \cite{maday2002priori,buffa2012priori}. This favorable behavior, however, relies on the assumption that the solution manifold can be well approximated by a low-dimensional linear space. While many diffusion-dominated elliptic and parabolic problems satisfy this property, the linear approximation paradigm becomes ineffective when the solution manifold exhibits strong nonlinear features, such as large curvature or moving coherent structures through the spatial domain. Such behaviors commonly arise in transport-dominated flows, wave propagation, and nonlinear evolution equations, where the Kolmogorov $N$-width decays slowly. Consequently, a global reduced basis requires a large number of basis functions to accurately capture all relevant solution features, leading to increased online computational cost. Both theoretical analyses and numerous numerical studies have demonstrated this limitation of linear reduced-order models, often referred to as the Kolmogorov barrier, thereby motivating the development of nonlinear model reduction techniques \cite{ohlberger2013reduced,ohlberger2016nonlinear,greif2019tensor,peherstorfer2022nonlinear,buffa2012priori}.

To address the limitations of linear model reduction, several nonlinear model reduction techniques have been developed \cite{hesthaven2026nonlinear}. A widely adopted strategy is localized model reduction, where the solution manifold is approximated by a collection of low-dimensional linear spaces rather than a single global space, together with an online mechanism that identifies the appropriate local approximation. Depending on how the local reduced spaces are constructed, localized methods can be broadly classified into several categories. One approach consists of partitioning the parameter domain into subregions, each associated with its own local reduced basis, allowing different approximation spaces to capture distinct parameter regimes \cite{haasdonk2008adaptive,eftang2010hp,haasdonk2011training,eftang2011hp2,guignard2024tree}. Another class of methods partitions the temporal domain into successive time windows, constructing dedicated reduced spaces for each interval to better represent transient dynamics \cite{dihlmann2011model,drohmann2011adaptive,shimizu2021windowed,copeland2022reduced}. Alternatively, localization can be performed directly on the solution manifold by grouping similar snapshots into clusters, typically through techniques such as $k$-means, and constructing a local reduced basis for each cluster \cite{amsallem2012nonlinear,washabaugh2012nonlinear,peherstorfer2014localized,amsallem2015fast,amsallem2016pebl}. In the present work, we focus on the first class of methods, namely parameter-domain partitioning based on hierarchical binary trees, owing to their adaptive offline construction, efficient online search, and amenability to rigorous convergence analysis.

In this article, we extend the tree-based localized reduced basis framework of \cite{barakat2026convergence}, originally developed for parametrized elliptic PDEs, to parametrized parabolic problems. The adaptive partitioning strategy introduced in \cite{barakat2026convergence} is retained, whereby the parameter domain is recursively subdivided into a hierarchy of subdomains, each associated with a local reduced basis. The construction of the local RB spaces is based on a POD/greedy algorithm of a prescribed number of parameters. Although the parameter-domain partitioning strategy remains unchanged, the extension to parabolic problems requires a careful treatment of the temporal approximation. In particular, the dimension of each local reduced basis space depends not only on the number of greedily selected parameter samples, but also on the number of POD modes retained for each selected parameter. It is therefore essential to establish a criterion for determining the number of POD modes to retain throughout the recursive partitioning process. To this end, we introduce a certified error estimator derived from an inf-sup stability analysis for parametrized parabolic problems. We establish that this estimator is equivalent to the underlying approximation error. Combined with suitable smoothness assumptions on the parameter-to-solution map, this equivalence enables the derivation of rigorous convergence results that characterize the number of parameter subdomains required to achieve a prescribed approximation accuracy. Building on this theoretical foundation, we further develop a principled strategy for selecting the number of POD modes associated with each greedy parameter. ensuring that an optimal convergence rate is maintained while avoiding an unnecessary increase in the dimension of the local reduced basis spaces.

The remainder of the paper is organized as follows. Section 2 introduces the parametrized parabolic problem and its discretization, and presents the preliminary stability results required for the subsequent analysis. In Section 3, we review the classical reduced basis methodology and introduce the proposed localized reduced basis approach. Section 4 is devoted to the convergence analysis of the proposed method. Finally, Section 5 presents numerical experiments for a parametrized time-dependent convection–diffusion problem, illustrating the effectiveness of the proposed approach and validating the theoretical results.

\section{Problem statement}
Let $\Omega \subset \mathbb{R}^P$, $P=1,2,3$, be a bounded Lipschitz domain, and let $\mathcal{D} \subset \mathbb{R}^d$ denote a tensor-product structured parameter domain. We consider the Hilbert spaces $V=V(\Omega)$ and $H=H(\Omega)$ such that $V \hookrightarrow H$ with dense embedding, yielding the Gelfand triple $V \hookrightarrow H \hookrightarrow V'$. 

Given a parameter $\mu \in \mathcal{D}$, we seek the solution $u(\cdot;\mu) \in L^2(0,T;V)\cap H^1(0,T;V')$ of the parameterized evolution problem
\begin{equation}
 \left\langle\partial_t u(t;\mu), v \right\rangle_{V'\times V} + a \left( u(t;\mu), v; \mu \right) = g(t) f(v;\mu), \quad \forall v \in V, \quad t\in (0,T],
 \label{cont}
\end{equation}
subject to an initial condition $u(0;\mu)=u_0(\mu) \in H$. Here, $g(t) \in C^0(0,T]$ is called the control function. In typical settings, one has $H_0^1(\Omega) \subset V\subset H^1(\Omega)$ and $H=L^2(\Omega)$. The bilinear form $a(\cdot,\cdot;\mu)$ is assumed to be uniformly continuous and coercive , and the linear functional $f(\cdot;\mu)$ is continuous uniformly on $V$. The existence and uniqueness of solutions to this problem are established in \cite{schwab2009space}.

To discretize the problem, we first divide the time interval into $M$ subintervals of equal length $\Delta t = T/M$ and define $t^m = m \Delta t$, $0\leq m \leq M$. Employing the backward Euler scheme, we introduce the backward difference quotient
\begin{equation}
    \delta u^m(\mu):= \frac{u^m(\mu) - u^{m-1}(\mu)}{\Delta t},    
\end{equation}
where $u^m(\mu)\approx u(t^m;\mu)$. For the spatial discretization, let $\mathcal{T}_h$ be a conforming mesh of $\Omega$ with mesh size $h$  and let $V_h \subset V$ be a finite element space defined on the mesh $\mathcal{T}_h$. We identify $V_h\equiv V_h'$ by means of the $H$-inner product. In particular, 
\begin{equation}
\langle  u_h,v_h\rangle_{V_h'\times V_h}
=
(u_h,v_h)_H,
\qquad \forall\,u_h,v_h\in V_h,
\end{equation}
and the dual norm on $V_h'$ is given by
\begin{equation}
\|u_h\|_{V_h'}
:=
\sup_{ v_h\in V_h}
\frac{(u_h,v_h)_H}
{\|v_h\|_V},
\qquad \forall u_h\in V_h.
\label{V_h'}
\end{equation}
The fully discrete counterpart of \eqref{cont} is formulated as follows:  Given $\mu \in \mathcal{D}$, for each time step $1 \le m \le M$, find $u_h^m(\mu) \in V_h$ such that
\begin{equation}
\left( \delta u_h^m(\mu), v_h \right)_H + a \left( u_h^m(\mu), v_h; \mu \right) = g(t^m) f(v_h;\mu), \quad \forall v_h \in V_h,
\label{eq:discrete}
\end{equation}
subject to the initial condition $( u_h^0(\mu), v_h)_H=( u_0(\mu), v_h)_H$, $\forall v_h\in V_h$.

For later analysis, it is convenient to rewrite \eqref{eq:discrete} in a global-in-time variational form. We therefore introduce the discrete trial space
\begin{equation}
    \mathcal{X}_h=\left\{u_h = (u_h^m)_{m=0}^M \ : \  u_h^m \in V_h,\ 0\leq m\leq M\right\},
\end{equation}
equipped with the discrete norm
\begin{equation}
\|u_h\|_{\mathcal{X}_h }^2 =\sum_{m=1}^M \Delta t \,\|\delta u_h^m \|_{V_h'}^2+ \sum_{m=1}^M \Delta t \,\|u_h^m\|_V^2 +\sum_{m=1}^M \Delta t^2 \,\|\delta u_h^m \|_{H}^2+\|u_h^M\|_H^2.
\end{equation}
The corresponding test space is $\mathcal Y_h=\mathcal{X}_h$, equipped with the discrete norm
\begin{equation}
\|v_h\|_{\mathcal{Y}}^2 = \|v_h^0\|_H^2+\sum_{m=1}^M \Delta t \,\|v_h^m\|_V^2.
\end{equation}
We define the bilinear form $b(\cdot, \cdot;\mu):\mathcal{X} \times \mathcal{Y} \xrightarrow{} \mathbb{R}$ by
\begin{equation}
b(u_h, v_h;\mu) = \sum_{m=1}^M \Delta t \left( \delta u_h^m, v_h^m\right )_H + \sum_{m=1}^M \Delta t \, a( u_h^m, v_h^m;\mu) +  ( u_h^0, v_h^0)_H.
\end{equation}
The discrete problem \eqref{eq:discrete} can then be equivalently written in the global form: 
Given $\mu\in\mathcal D$, find $u_h(\mu)\in\mathcal{X}_h$ such that
\begin{equation}
    b(u_h(\mu), v_h;\mu) = \sum_{m=1}^M \Delta t\, g(t^m) f(v_h^m;\mu) + ( u_0(\mu), v_h^0)_H \quad \forall v_h\in\mathcal Y_h.
    \label{glob-dis}
\end{equation}

The discrete coercivity and continuity constants are defined by
\begin{equation}
    \alpha_h(\mu):=\inf _{v_h \in V_h } \frac{a(v_h, v_h ; \mu)}{\|v_h\|_{V}^2}, \qquad  \gamma_h(\mu):=\sup _{w_h \in V_h } \sup _{v_h \in V_h } \frac{a(w_h, v_h ; \mu)}{\|w_h\|_V \|v_h\|_{V}}.
\end{equation}
Furthermore, we assume that computable bounds $\alpha_{\mathrm{LB}}(\mu)$ and $\gamma_{\mathrm{UB}}(\mu)$ are available such that
\begin{align}
    \alpha_h(\mu)\geq \alpha_{\mathrm{LB}}(\mu) \geq \underline{\alpha} > 0, \quad \forall \mu \in \mathcal{D},\\
    \gamma_h(\mu)\leq  \gamma_{\mathrm{UB}}(\mu) \leq \overline{\gamma} < \infty, \quad \forall \mu \in \mathcal{D},
\end{align}
where $\underline{\alpha}$ and $\overline{\gamma}$ are the uniform coercivity and continuity constants. To enable an efficient offline–online decomposition, we assume that the bilinear form $a(\cdot,\cdot;\mu)$, the linear functional $f(\cdot;\mu)$ and the initial condition admit the affine representations
\begin{align}
a(w,v;\mu)
&= \sum_{q=1}^{Q_{\mathrm{a}}} \theta_{\mathrm{a}}^{q}(\mu) \, a_q(w,v),
\label{aff 1}
\\
f(v;\mu)
&= \sum_{q=1}^{Q_{\mathrm{f}}} \theta_{\mathrm{f}}^{q}(\mu) \, f_q(v),
\label{aff 2}
\\
u_0(\mu)&=\sum_{q=1}^{Q_{0}} \theta_0^q(\mu) u_q,
\label{i.c aff}
\end{align}
where the components $a_q(\cdot,\cdot)$, $f_q(\cdot)$ and $u_q$ are independent of the parameter $\mu$, while the coefficient functions $\theta_{\mathrm{a}}^{q}(\mu)$, $\theta_{\mathrm{f}}^{q}(\mu)$ and $\theta_0^q(\mu)$ depend continuously on $\mu$.
If such an affine decomposition is not available in closed form, it can be approximated to high accuracy using the Empirical Interpolation Method \cite{barrault2004empirical, grepl2007efficient}.

\subsection{Preliminary results}
The analysis presented in this section provides the stability framework for the discrete formulation \eqref{glob-dis}. We first prove the continuity of the associated bilinear form and then establish a discrete inf--sup condition with respect to the norms $\|\cdot\|_{\mathcal{X}_h}$ and $\|\cdot\|_{\mathcal{Y}}$. These properties imply the well-posedness of the discrete problem and constitute the the basis for the quasi-optimal error analysis of the reduced basis approximation. We begin by deriving a bound for the discrete initial condition in terms of the $\mathcal{X}$-norm. To this end, we make use of the following identity:
\begin{equation}
2\Delta t \left( \delta u_h^m, u_h^m \right)_H
=
\|u_h^m\|_H^2
-
\|u_h^{m-1}\|_H^2
+
\Delta t^2 \|\delta u_h^m\|_H^2,
\qquad 1\le m\le M.
\end{equation}
Summing over $m=1,\ldots,M$ yields
\begin{equation}
2\sum_{m=1}^M \Delta t \left( \delta u_h^m, u_h^m \right)_H
=
\|u_h^M\|_H^2
-
\|u_h^0\|_H^2
+
\sum_{m=1}^M \Delta t^2 \|\delta u_h^m\|_H^2.
\label{dis energy}
\end{equation}
Rearranging, we obtain
\begin{equation}
\|u_h^0\|_H^2
=
\|u_h^M\|_H^2
+
\sum_{m=1}^M \Delta t^2 \|\delta u_h^m\|_H^2
- 
2\sum_{m=1}^M \Delta t \left( \delta u_h^m, u_h^m \right)_H.
\end{equation}
Using the duality pairing estimate followed by Young's inequality, $2ab\le a^2+b^2,$ gives
\begin{equation}
2 \left| \left( \delta u_h^m, u_h^m \right)_H \right|
\le
\|\delta u_h^m\|_{V_h'}^2
+
\|u_h^m\|_V^2.
\end{equation}
Therefore,
\begin{equation}
\|u_h^0\|_H^2
\le
\|u_h^M\|_H^2
+
\sum_{m=1}^M \Delta t^2 \|\delta u_h^m\|_H^2
+
\sum_{m=1}^M \Delta t \|\delta u_h^m\|_{V_h'}^2 
+ 
\sum_{m=1}^M \Delta t \|u_h^m\|_V^2.
\end{equation}
Hence,
\begin{equation}
\|u_h^0\|_H \le \|u_h\|_{\mathcal{X}_h}.
\label{initial_bound}
\end{equation}

It follows immediately that the bilinear form $b(\cdot,\cdot;\mu): \mathcal{X}_h \times \mathcal Y_h \to \mathbb R$ is continuous, with continuity constant \(\gamma_b(\mu)\) satisfying
\begin{equation}
|b(w,v;\mu)|\le \gamma_b(\mu)\, \|w\|_{\mathcal{X}_h}\, \|v\|_{\mathcal Y}, \qquad \forall\, w \in \mathcal{X}_h,\; v \in \mathcal Y_h,
\end{equation}
where 
\begin{equation}
    \gamma_b(\mu)=\gamma_h(\mu)+2\leq \overline{\gamma}+2=:\overline{\gamma_b}.
\end{equation}

\begin{proposition}[Discrete inf--sup condition]
Let $\mu \in \mathcal{D}$. Then, there exists a constant $\beta_h(\mu) > 0$, such that
\begin{equation}
\inf_{u_h \in \mathcal{X}_h} \sup_{v_h \in \mathcal{Y}_h} \frac{|b(u_h, v_h;\mu)|}{\|u_h\|_{\mathcal{X}_h} \|v_h\|_{\mathcal{Y}}} \geq \beta_h(\mu).
\label{infsup}
\end{equation}
\label{inf sup prop}
\end{proposition}

\begin{proof}
We define the discrete operator $A_h(\mu):V_h\to V_h'$ by
\begin{equation}
\langle A_h(\mu)u_h,v_h\rangle_{V_h'\times V_h}
=
a(u_h,v_h;\mu),
\qquad \forall\,u_h,v_h\in V_h.
\end{equation}
Its adjoint $A_h'(\mu):V_h\to V_h'$ is defined by
\begin{equation}
\langle A_h'(\mu)u_h,v_h\rangle_{V_h'\times V_h}
=
a(v_h,u_h;\mu),
\qquad \forall\,u_h,v_h\in V_h.
\end{equation}
From the coercivity and boundedness of $a(\cdot,\cdot;\mu)$, we obtain
\begin{equation}
\|A_h'(\mu)\|_{V_h \to V_h'} \leq \gamma_h(\mu), \quad \|(A_h'(\mu))^{-1}\|_{V_h' \to V_h} \leq \frac{1}{\alpha_h(\mu)}.
\end{equation}

Let $u_h \in \mathcal{X}_h$ be arbitrary. For each time step $m=1,\dots,M$, define
\begin{equation}
z_h^m = A_h'(\mu)^{-1} \delta u_h^m \in V_h.
\end{equation}
We then choose the test function $v_h \in \mathcal{Y}_h$ as
\begin{align}
    v_h^m &= z_h^m + u_h^m, \quad m = 1, \ldots, M, \\
    v_h^0 &= u_h^0.
\end{align}
From the discrete operator bounds, we obtain
\begin{equation}
\|z_h^m\|_V \leq \frac{1}{\alpha_h(\mu)} \left\| \delta u_h^m \right\|_{V_h'}.
\end{equation}
Hence,
\begin{equation}
\|v_h^m\|_V \leq \|z_h^m\|_V + \|u_h^m\|_V \leq \frac{1}{\alpha_h(\mu)} \left\| \delta u_h^m \right\|_{V_h'} + \|u_h^m\|_V.
\end{equation}
Applying the inequality $(a + b)^2 \leq 2(a^2 + b^2)$, and Summing with $\Delta t$ weights yields
\begin{equation}
\sum_{m=1}^{M} \Delta t \|v_h^m\|_V^2 \leq 2 \left( \frac{1}{\alpha_h(\mu)^2} \sum_{m=1}^{M} \Delta t \left\| \delta u_h^m \right\|_{V_h'}^2 + \sum_{m=1}^{M} \Delta t \|u_h^m\|_V^2 \right).
\end{equation}
Combining this with \eqref{initial_bound}, we obtain
\begin{equation}
\|v_h\|_{\mathcal{Y}}^2 \leq C_1 \|u_h\|_{\mathcal{X}_h}^2,
\end{equation}
 with
 \begin{equation}
     C_1 = 2 \max \left( \frac{1}{\alpha_h(\mu)^2}, 1 \right) +1.
 \end{equation}

Next, evaluating $b(u_h,v_h;\mu)$ yields
\begin{equation}
\begin{aligned}
  b(u_h, v_h; \mu) =& \sum_{m=1}^M \Delta t \Bigl[ \left( \delta u_h^m, z_h^m \right)_H + \left( \delta u_h^m, u_h^m \right)_H \\
  &+ a(u_h^m, z_h^m; \mu) + a(u_h^m, u_h^m; \mu) \Bigr] + \|u_h^0\|_H^2.
\end{aligned}
\end{equation}
For the first term on the right-hand side, we have
\begin{equation}
\begin{aligned}
  \left( \delta u_h^m, z_h^m \right)_H &= ( A_h'(\mu) z_h^m, z_h^m )_H \\
  &= a(z_h^m, z_h^m; \mu) \geq \alpha_h(\mu) \|z_h^m\|_V^2 \geq \frac{\alpha_h(\mu)}{\gamma_h(\mu)^2} \left\| \delta u_h^m \right\|_{V_h'}^2.
\end{aligned}
\end{equation}
By the definition of the discrete adjoint, we have
\begin{equation}
a(u_h^m, z_h^m; \mu) = ( A_h'(\mu) z_h^m, u_h^m )_H = \left( \delta u_h^m, u_h^m \right)_H.
\end{equation}
Therefore, the second and third terms combine as
\begin{equation}
a(u_h^m, z_h^m; \mu) + \left( \delta u_h^m, u_h^m \right)_H = 2 \left( \delta u_h^m, u_h^m \right)_H.
\end{equation}

Using the identity \eqref{dis energy} and the coercivity of $a(u_h^m, u_h^m; \mu)$, we obtain
\begin{equation}
\begin{aligned}
  b(u_h, v_h; \mu) \geq  \sum_{m=1}^M \Delta t \left( \frac{\alpha_h(\mu)}{\gamma_h(\mu)^2} \left\| \delta u_h^m \right\|_{V_h'}^2 + \alpha_h(\mu) \|u_h^m\|_V^2 \right) 
  &+ \|u_h^M\|_H^2 + \sum_{m=1}^M \Delta t ^2 \,\|\delta u_h^m \|_H^2.
\end{aligned}
\end{equation}
Therefore,
\begin{equation}
  b(u_h, v_h; \mu) \geq \alpha_b(\mu) \|u_h \|^2_{\mathcal{X}_h},
\end{equation}
where 
\begin{equation}
    \alpha_b(\mu) = \min \left( \frac{\alpha_h(\mu)}{\gamma_h(\mu)^2}, \alpha_h(\mu),1 \right) > 0.
\end{equation}
Therefore,
\begin{equation}
\frac{|b(u_h, v_h; \mu)|}{\|u_h \|_{\mathcal{X}_h} \| v_h \|_{\mathcal{Y}}} \geq \frac{\alpha_b(\mu)}{\sqrt{C_1}} > 0.
\end{equation}
Since $u_h \in \mathcal{X}_h$ is chosen arbitrarily, we obtain \eqref{infsup} with 
\begin{equation}
\beta_h(\mu) \geq \frac{\alpha_b(\mu)}{\sqrt{C_1}} = \frac{\min\left(\frac{\alpha_h(\mu)}{\gamma_h(\mu)^2}, \alpha_h(\mu),1\right)}{\sqrt{2\max\left(\frac{1}{\alpha_h(\mu)^2}, 1\right)+1}} > 0.
\end{equation}
\end{proof}

Using the uniform bounds on $\alpha_h(\mu)$ and $\gamma_h(\mu)$, we obtain the parameter-uniform estimate
\begin{equation}
\beta_h(\mu) \geq \beta_{\mathrm{LB}}(\mu) \geq \underline{\beta} > 0, \quad \forall \mu \in \mathcal{D},
\end{equation}
such that
\begin{equation}
\beta_{\mathrm{LB}}(\mu) := \frac{\min\left(\frac{\alpha_{\mathrm{LB}}(\mu)}{\gamma_{\mathrm{UB}}(\mu)^2}, \alpha_{\mathrm{LB}}(\mu),1\right)}{\sqrt{2\max\left(\frac{1}{\alpha_{\mathrm{LB}}(\mu)^2}, 1\right) +1 }} > 0,
\end{equation}
and
\begin{equation}
\underline{\beta} := \frac{\min\left(\frac{\underline{\alpha}}{\overline{\gamma}^2}, \underline{\alpha},1\right)}{\sqrt{2\max\left(\frac{1}{\underline{\alpha}^2}, 1\right) +1}} > 0.
\end{equation}

\begin{corollary}[Discrete energy estimate]
Let $\mu \in \mathcal{D}$, and let $u_h(\mu) \in \mathcal{X}_h$ be the solution of \eqref{glob-dis}. Then, there exists a constant $C > 0$, such that
\begin{equation}
\sup_{\mu \in \mathcal{D}} \|u_h(\mu)\|_{\mathcal{X}_h} \leq C \sup_{\mu \in \mathcal{D}} \left( \|f(\mu)\|_{V'}^2 \, \|g\|_{L^2(0,T)}^2 + \|u_h^0(\mu)\|_H^2 \right)^{1/2}.
\end{equation}
\label{uni bound}
\end{corollary}

\begin{proof}
From the discrete inf--sup condition \eqref{infsup}, we have
\begin{align}
\|u_h(\mu)\|_{\mathcal{X}_h} &\leq \frac{1}{\beta_h(\mu)} \sup_{v_h \in \mathcal{Y}_h} \frac{|b(u_h(\mu), v_h; \mu)|}{\|v_h\|_{\mathcal{Y}}} \\
&= \frac{1}{\beta_h(\mu)} \sup_{v_h \in \mathcal{Y}_h} \frac{| \sum_{m=1}^M \Delta t \, g(t^m) f(v_h^m;\mu) + ( u_h^0(\mu), v_h^0 )_H |}{\|v_h\|_{\mathcal{Y}}} \\
&\leq \frac{1}{\beta_h(\mu)} \sup_{v_h \in \mathcal{Y}_h} \frac{ \sum_{m=1}^M \Delta t \, |g(t^m)| \|f(\mu)\|_{V'} \|v_h^m\|_V + \|u_h^0(\mu)\|_H \|v_h^0\|_H }{ \|v_h\|_{\mathcal{Y}}}.
\end{align}

Applying the discrete Cauchy–Schwarz inequality yields
\begin{align}
\|u_h(\mu)\|_{\mathcal{X}_h} &\leq \frac{1}{\beta_h(\mu)} \sup_{v_h \in \mathcal{Y}_h} \frac{ \|f(\mu)\|_{V'} \left( \sum_{m=1}^M \Delta t \, |g(t^m)|^2 \right)^{1/2} \left( \sum_{m=1}^M \Delta t \, \|v_h^m\|_V^2 \right)^{1/2} + \|u_h^0(\mu)\|_H \|v_h^0\|_H }{\|v_h\|_{\mathcal{Y}}} \\
&\leq \frac{1}{\beta_h(\mu)} \sup_{v_h \in \mathcal{Y}_h} \frac{ \left( \|f(\mu)\|_{V'}^2 \sum_{m=1}^M \Delta t \, |g(t^m)|^2 + \|u_h^0(\mu)\|_H^2 \right)^{1/2} \|v_h\|_{\mathcal{Y}} }{\|v_h\|_{\mathcal{Y}}}.
\end{align}
Therefore,
\begin{align}
\|u_h(\mu)\|_{\mathcal{X}_h} \leq C \left( \|f(\mu)\|_{V'}^2 \, \|g\|_{L^2(0,T)}^2 + \|u_h^0(\mu)\|_H^2 \right)^{1/2},
\end{align}
with
\begin{equation}
C := \sup_{\mu \in \mathcal{D}} \frac{1}{\beta_h(\mu)} \leq \frac{1}{\underline{\beta}}.
\end{equation}
\end{proof}

\section{The Reduced Basis Method}
The RB method aims to construct a low-dimensional subspace \(\mathcal{X}_N \subset \mathcal{X}_h\) that accurately approximates the solution manifold across the entire parameter domain \(\mathcal{D}\). Here, $N$ represents the number of parameters involved in the construction the RB space. By projecting the fully discrete problem onto this reduced space, we obtain a reduced-order model that can be solved efficiently in the online stage, provided the affine decomposition (\ref{aff 1}-\ref{i.c aff}) enables an offline–online computational decomposition.

We define the discrete solution manifold as
\begin{equation}
    \mathcal{M}_h=\mathcal{M}_h(\mathcal{D}):=\left\{u_h(\mu) = (u_h^m(\mu))_{m=0}^M \mid \mu \in \mathcal{D}\right\} \subset \mathcal{X}_h.
\end{equation}
The RB trial space is defined as
\begin{equation}
    \mathcal{X}_N=\left\{u_N = (u_N^m)_{m=0}^M \ : \ u_N^m \in V_N,\ 0\leq m\leq M\right\},
\end{equation}
and is equipped with the norm 
\begin{equation}
\|u_h\|_{\mathcal{X}_N }^2 =\sum_{m=1}^M \Delta t \,\|\delta u_N^m \|_{V_N'}^2+ \sum_{m=1}^M \Delta t \,\|u_N^m\|_V^2 +\sum_{m=1}^M \Delta t^2 \,\|\delta u_N^m \|_{H}^2+\|u_N^M\|_H^2,
\end{equation}
where 
\begin{equation}
\|\delta u_N^m\|_{V_N'}
:=
\sup_{ v_N\in V_N}
\frac{(\delta u_N^m,v_N)_H}
{\|v_N\|_V}.
\label{V_N'}
\end{equation}
The corresponding RB test space is \(\mathcal{Y}_N=\mathcal{X}_N \) equipped with the norm $\|\cdot\|_\mathcal{Y}$. The RB counterpart of \eqref{glob-dis} is given by: For any \(\mu\in\mathcal D\), find \(u_N(\mu)\in\mathcal{X}_N\) such that
\begin{equation}
    b(u_N(\mu), v_N;\mu) = \sum_{m=1}^M \Delta t\, g(t^m) f(v_N^m;\mu) + ( u_0(\mu), v_N^0)_H  \quad \forall v_N\in\mathcal Y_N. 
    \label{glob-rb}
\end{equation}
Equivalently, the reduced problem can be written in time-stepping form: for each $1 \le m \le M$, find $u_N^m(\mu) \in V_N$ such that
\begin{equation}
\left( \delta u_N^m(\mu), v_N \right)_H + a \left( u_N^m(\mu), v_N; \mu \right) = g(t^m) f(v_N;\mu), \quad \forall v_N \in V_N, \quad 1\leq m \leq M,
\label{eq:rb}
\end{equation}
subject to the initial condition $( u_N^0(\mu), v_N)_H=( u_0(\mu), v_N)_H$, $\forall v_N\in V_N$. Proceeding exactly as in Proposition~\ref{inf sup prop}, one obtains the following discrete inf--sup stability result at the reduced basis level: there exists a constant $\beta_N(\mu)>0$ such that
\begin{equation}
\inf_{u_N \in \mathcal{X}_N} \sup_{v_N \in \mathcal{Y}_N} \frac{|b(u_N, v_N;\mu)|}{\|u_h\|_{\mathcal{X}_N} \|v_N\|_{\mathcal{Y}}} \geq \beta_N(\mu),
\label{RB insfup}
\end{equation}
where
\begin{equation}
\beta_N(\mu) \geq \beta_h(\mu) \geq \beta_{\mathrm{LB}}(\mu) \geq \underline{\beta} > 0, \quad \forall \mu \in \mathcal{D},
\end{equation}

The next lemma provides a quasi-optimality estimate for the reduced basis approximation. In particular, it shows that the RB solution is, up to a constant factor, as accurate as the best approximation from the reduced space.

\begin{lemma}[Quasi-optimality]
Let $\mu \in \mathcal{D}$, and let $u_h(\mu) \in \mathcal{X}_h$ and $u_N(\mu) \in \mathcal{X}_N$ denote the solutions of \eqref{glob-dis} and \eqref{glob-rb}, respectively. Then,
\begin{equation}
\|u_h(\mu) - u_N(\mu)\|_{\mathcal{X}_N} \leq \left(1 + \frac{\gamma_b(\mu)}{\beta_N(\mu)}\right) \inf_{v_N \in \mathcal{X}_N} \|u_h(\mu) - v_N\|_{\mathcal{X}_N}.
\label{optimal}
\end{equation}
\label{qoptimal}
\end{lemma}

\begin{proof}
Let $v_N \in \mathcal{X}_N$. By the discrete inf--sup condition \eqref{RB insfup}, we have
\begin{align}
\|u_N(\mu) - v_N\|_{\mathcal{X}_N} &\leq \frac{1}{\beta_N(\mu)} \sup_{w_N \in \mathcal{Y}_N} \frac{|b(u_N(\mu) - v_N, w_N; \mu)|}{\|w_N\|_{\mathcal{Y}}}.
\end{align}
Using linearity and Galerkin orthogonality, we obtain
\begin{align}
\|u_N(\mu) - v_N\|_{\mathcal{X}_N} &\leq \frac{1}{\beta_N(\mu)} \sup_{w_N \in \mathcal{Y}_N} \frac{|b(u_h(\mu) - v_N, w_N; \mu)|}{\|w_N\|_{\mathcal{Y}}} \\
&\leq \frac{1}{\beta_N(\mu)} \sup_{w_N \in \mathcal{Y}_N} \frac{\gamma_b(\mu) \|u_h(\mu) - v_N\|_{\mathcal{X}_N} \|w_N\|_{\mathcal{Y}}}{\|w_N\|_{\mathcal{Y}}} \\
&= \frac{\gamma_b(\mu)}{\beta_N(\mu)} \|u_h(\mu) - v_N\|_{\mathcal{X}_N}.
\end{align}
Finally, applying the triangle inequality gives
\begin{align}
\|u_h(\mu) - u_N(\mu)\|_{\mathcal{X}_N} &\leq \|u_h(\mu) - v_N\|_{\mathcal{X}_N} + \|u_N(\mu) - v_N\|_{\mathcal{X}_N}  \\
&\leq \left(1 + \frac{\gamma_b(\mu)}{\beta_N(\mu)}\right) \|u_h(\mu) - v_N\|_{\mathcal{X}_N}.
\end{align}
Since $v_N \in \mathcal{X}_N$ was arbitrary, taking the infimum over all $v_N \in \mathcal{X}_N$ yields \eqref{optimal}.
\end{proof}

The estimate \eqref{optimal} is expressed in the reduced basis norm $\|\cdot\|_{\mathcal X_N}$, in which the time derivative is measured in the reduced dual norm $V_N'$. For the subsequent analysis, it is more convenient to measure the error in the high-fidelity norm $\|\cdot\|_{\mathcal X_h}$. This requires comparing the dual norms on $V_h$ and $V_N$. From the definitions of the dual norms \eqref{V_h'} and
\eqref{V_N'}, it follows immediately that
\begin{equation}
    \|v_n\|_{V_N'} \leq \|v_n\|_{V_h'}, \quad \forall v_n \in V_N.
\end{equation}
To obtain the converse estimate, we use the inverse inequality on the finite element space $V_h$,
\begin{equation}
\|v_h\|_V \le C_{\mathrm{inv}}h^{-1}\|v_h\|_H,
\qquad \forall\,v_h\in V_h.
\label{normeq}
\end{equation}
let $P_N^H:V_h\rightarrow V_N$ denote the $H$-orthogonal projection.
Then, for every $v_N\in V_N$,
\begin{align}
\|v_N\|_{V_h'}
&= \sup_{w_h\in V_h} \frac{(v_N,w_h)_H}{\|w_h\|_V} = \sup_{w_h\in V_h} \frac{(v_N,P_N^Hw_h)_H}{\|w_h\|_V} \\
&\le \sup_{w_h\in V_h} \frac{\|v_N\|_{V_N'}\,\|P_N^Hw_h\|_V} {\|w_h\|_V} \le \sup_{w_h\in V_h} \frac{C_{\mathrm{inv}}h^{-1}\|v_N\|_{V_N'}\,\|P_N^Hw_h\|_H} {\|w_h\|_V}\\
&\le \sup_{w_h\in V_h} \frac{C_{\mathrm{inv}}h^{-1}\|v_N\|_{V_N'}\,\|w_h\|_H} {\|w_h\|_V}\le C_{\mathrm{inv}}h^{-1} C_P \|v_N\|_{V_N'},
\label{dual-norm-2}
\end{align}
where $C_P$ is the Poincare constant. Consequently, for any $V_N\subset V_h$, the dual norms are equivalent on $V_N$, i.e.,
\begin{equation}
\|v_N\|_{V_N'} \le \|v_N\|_{V_h'} \le C_{\mathrm{inv}}h^{-1} C_P \|v_N\|_{V_N'},
\qquad
\forall\,v_N\in V_N.
\label{dual-equivalence}
\end{equation}
Since the norms $\|\cdot\|_{\mathcal X_h}$ and $\|\cdot\|_{\mathcal X_N}$ differ only in the dual norm used to measure the time derivative, \eqref{dual-equivalence} yields
\begin{equation}
\|w_N\|_{\mathcal X_h} \le \max\{1,C_{\mathrm{inv}}h^{-1}C_P\} \|w_N\|_{\mathcal X_N},
\qquad \forall\,w_N\in\mathcal X_N.
\label{X-norm-equiv}
\end{equation}
Combining \eqref{optimal} with \eqref{X-norm-equiv} gives
\begin{equation}
\|u_h(\mu)-u_N(\mu)\|_{\mathcal X_h} \le C_h \inf_{v_N\in\mathcal X_N} \|u_h(\mu)-v_N\|_{\mathcal X_N},
\label{quais_X_h}
\end{equation}
where
\begin{equation}
    C_h = \max\{1,C_{\mathrm{inv}}h^{-1}C_P\}\left(1+\frac{\overline{\gamma_b}}{\underline{\beta}}\right).
\end{equation}
is a constant depending only on the high-fidelity space.

By repeating the argument used to establish Corollary~\ref{uni bound} at the reduced basis level, and subsequently invoking the norm equivalence \eqref{X-norm-equiv}, we obtain the following uniform stability estimate for the reduced basis solution
\begin{equation}
\sup_{\mu \in \mathcal{D}} \|u_N(\mu)\|_{\mathcal{X}_h} \leq \frac{\max\{1,C_{\mathrm{inv}}h^{-1}C_P\}}{\underline{\beta}} \sup_{\mu \in \mathcal{D}} \left( \|f(\mu)\|_{V'}^2 \, \|g\|_{L^2(0,T)}^2 + \|u_N^0(\mu)\|_H^2 \right)^{1/2}.
\label{RB uni bound}
\end{equation}

\subsection{POD/Greedy sampling}
The reduced basis space \(V_N\) is constructed using the the POD/Greedy algorithm which we briefly review below. We begin by recalling the fundamental POD optimality property. Given   $(w^m)_{m=0}^M \in \mathcal{X}_h$, the procedure $\mathrm{POD}(\{w^m \in V_h,\, 1 \le m \le M\}, i)$ returns $i$  functions \{$\xi_j, 1 \le j \le i\}\subset V_h$ that are orthonormal with respect to the inner product $(\cdot,\cdot)_V$. The associated POD space $ V_{\mathrm{POD}} = \mathrm{span}\{\xi_1,\ldots,\xi_i\}$ is characterized as the solution of the minimization problem

\begin{equation}
V_{\mathrm{POD}} = \arg\inf_{V \subset \mathrm{span}\{w^m(\mu),\, 1\le m \le M\}}
\left( \frac{1}{M}\sum_{m=1}^M \inf_{v\in V} \|w^m(\mu)-v\|_V^2 \right)^{1/2}.
\end{equation}

\begin{algorithm}
\caption{$[\xi_1,\ldots,\xi_i] = \mathrm{POD}(\{w^m \in V_h,\, 1 \le m \le M\},i) $}
\label{alg:POD}
\begin{algorithmic}[1]
\State Compute the correlation matrix $ C_{mq} = \frac{1}{M}(w^m,w^q)_V, \qquad 1 \le m,q \le M.$
\State Compute the $i$ largest eigenvalues $\lambda_j$ and corresponding eigenvectors $\psi_j \in \mathbb{R}^M$ satisfying $C \psi_j = \lambda_j \psi_j.$
\State For $j=1,\ldots,i$, define the POD modes by $ \xi_j = \left(\sum_{m=1}^M (\psi_j)_m \, w^m \right) \bigg/ \,\left\|\sum_{m=1}^M (\psi_j)_m \, w^m\right\|_V.$
\State Return the $V$-orthonormal basis $\{\xi_j\}_{j=1}^i$.
\end{algorithmic}
\end{algorithm}

The greedy algorithm is an iterative search procedure for constructing low-dimensional approximation spaces. At each iteration, the parameter corresponding to the largest value of a prescribed error estimator is selected, and the associated solution is added to the current approximation space. To describe the basic idea,  let $F:=\{f(\mu):\mu \in \mathcal{D}\}\subset\mathcal{X}_h$ be a compact set of parametrized functions. The goal is to construct, for each $N \ge 1$, an $N$-dimensional subspace $F_N=\text{span} \{f_1, \dots , f_N\}$ that provides an accurate approximation of the set $F$. Let $P_N$ be the $\mathcal{X}$--orthogonal projector onto $F_N$. Then, the greedy approximation error is defined as
\begin{equation}
    \sigma_N(F):= \sup_{f \in F} \|f - P_N f\|_{\mathcal{X}_h}
    \label{sigma_n}
\end{equation}
which quantifies the worst-case error incurred when approximating elements of $F$ by their orthogonal projections onto $F_N$.

A fundamental result in greedy approximation theory states that the greedy error is controlled by the Kolmogorov $N$--width of the set $F$; see, for example, \cite{binev2011convergence, buffa2012priori}. Specifically,
\begin{equation}
    \sigma_N(F) \leq \kappa d_N(F),
    \label{th comp}
\end{equation}
where $\kappa$ is a constant that depends on $N$ and the effectivity index of error estimator, and the Kolmogorov $N$--width, defined as 
\begin{equation}
    d_N(F):= \inf_{\substack{W \subset \mathcal{X}_h \\ \dim(W)=N}} \sup _{u \in F} \inf _{w \in W}\|u-w\|_{\mathcal{X}_h},
\end{equation}
represents the smallest possible worst-case approximation error achievable by any $N$-dimensional subspace of $\mathcal{X}_h$. 

It should be noted that, in the present context, each element of $F$ is a trajectory $ f(\mu) = \bigl(f^m(\mu)\bigr)_{m=0}^M \in \mathcal{X}_h,$ where each snapshot $f^m(\mu)$ belongs to $V_h$. Consequently, the reduced space $F_N$ has dimension $N$ as a subspace of $\mathcal{X}_h$. Since each basis element represents an entire trajectory, the associated collection of spatial snapshots contains up to $N(M+1)$ vectors in $V_h$. When the initial condition admits the affine representation \eqref{i.c aff}, the initial snapshots need not be included for each trajectory. Instead, it suffices to include, only once, the $H$-projections onto $V_h$ of the affine components $\{u_q\}_{q=1}^{Q_0}$. Consequently, the total number of snapshot vectors in $V_h$ is at most $NM+Q_0$.

The POD/Greedy algorithm introduced in \cite{haasdonk2008reduced} combines a greedy sampling strategy in the parameter domain with POD of the time trajectories. At each iteration $n, 1\leq n\leq N$, the greedy procedure selects a parameter value $\mu^{(n)} \in E$, where $E \subset \mathcal{D}$ is a finite training set. The corresponding high-fidelity trajectory $\{u_h^m(\mu^{(n)})\}_{m=0}^M$ is then computed. We then evaluate the projection error
\begin{equation}
    e^m(\mu^{(n)}) = u_h^m(\mu^{(n)}) - P_{V_{n-1}}\bigl(u_h^m(\mu^{(n)})\bigr),
\qquad 1 \le m \le M,
\end{equation}
where $P_{V_{n-1}}$ denotes the projection onto the current reduced space $V_{n-1}$ with respect to the $V$-inner product. The number of modes that satisfy a prescribed criterion is first evaluated. The POD algorithm is then applied to the resulting error trajectory to extract these dominant modes, which are subsequently added to the reduced space.

Once the reduced space has been updated, the next parameter sample is selected as the maximizer of an a posteriori error estimator $\eta_n^M(\mu)$ over the training set $E$. The construction of the estimator $\eta_n^M(\mu)$ is discussed in detail below. The algorithm ultimately returns the reduced basis space $V_N$ together with the maximum error estimator over the training set.

To initiate Algorithm~\ref{alg:PODGreedy} , we specify the maximum number of parameter samples $N$ used to construct the RB space, a training set $E\subset \mathcal{D}$, an initial parameter value $\mu^{(1)}$, and a tolerance $\epsilon$. The maximum estimated error is initialized as $\epsilon_{\max}=\infty$, and the reduced basis space is initialized by $V_0=\operatorname{span} \{u_q^h,\leq q \leq Q\}$, where $u_q^h\in V_h$ denotes the finite element projection of the initial-condition components $u_q$.

An alternative variant of the POD--Greedy algorithm was proposed in \cite{hesthaven2016certified}, where a two-stage POD compression strategy is employed to avoid storing and processing the full error trajectory. A purely greedy strategy in both the temporal and parametric domains, although often effective in practice, may stagnate and fail to achieve the same level of approximation efficiency; see \cite{grepl2005posteriori}.

\begin{algorithm}
\caption{$[V_N,\epsilon_{\max}]=\mathrm{POD/Greedy}(N,\epsilon,\mu^{(1)},E)$}
\label{alg:PODGreedy}
\begin{algorithmic}[1]
\State Set $V=\operatorname{span}\{u_q^h,\leq q \leq Q\}$, $\epsilon_{\max}=\infty$, $n=0$.
\While{$n \leqslant N$ and $\epsilon_{\max} > \epsilon$}
    \State $n \gets n + 1$.
    \State Compute $u_h^m(\mu^{(n)})$, $1 \leq m \leq M$.
    \State $e^m(\mu^{(n)}) \gets u_h^m(\mu^{(n)}) - P_V\big(u_h^m(\mu^{(n)})\big)$, $1 \leq m \leq M$.
    \State Specify the number of POD modes $i_n$.
    \State $[\xi_1,\ldots,\xi_i]=\mathrm{POD}(\{e^m(\mu^{(n)}) \in V_h,\, 1 \le m \le M\},i_n)$.
    \State $V_n \gets V_{n-1} \oplus \operatorname{span}\{\xi_j,\; 1 \leq j \leq i_n\}$.
    \State $\mu^{(n+1)} \gets \argmax_{\mu \in E} \eta_n^M(\mu)$.
    \State $\epsilon_{\max} \gets \max_{\mu \in E} \eta_n^M(\mu)$.
\EndWhile
\end{algorithmic}
\end{algorithm}

\subsection{Error estimator}
An efficient \textit{a posteriori} error estimator is a crucial ingredient in the construction of the reduced basis space. The two key ingredients are the dual norm of the reduced basis residual and a computable lower bound for the discrete inf--sup constant $\beta_h(\mu)$. It follows from the problem statements for $u_h(\mu)$, defined by \eqref{eq:discrete}, and $u_N(\mu)$, defined by \eqref{eq:rb}, that 
\begin{equation}
    \left( \delta  (u_h^m(\mu)-u_N^m(\mu)),v_h \right)_H + a\big( u_h^m(\mu)-u_N^m(\mu),v_h;\mu\big)=r_N^m(v_h;\mu), \quad \forall v_h \in V_h, \quad 1\leq m \leq M,
    \label{res error}
\end{equation}
where $r_N^m(\cdot;\mu)\in V_h'$ is the residual, 
\begin{equation}
r_N^m(v_h;\mu)
=
g(t^m)f(v_h;\mu)
-
\left( \delta u_N^m(\mu), v_h \right)_H
-
a\bigl(u_N^m(\mu), v_h;\mu\bigr),
\qquad
\forall v_h \in V_h.
\label{err}
\end{equation}
Let $R_N^m(\mu)$ be the Riesz representer of the residual, such that
\begin{equation}
\|R_N^m(\mu)\|_V = \sup_{v_h \in V_h} \frac{|r_N^m(v_h; \mu)|}{\|v\|_V}.
\end{equation}
We define the error estimate
\begin{equation}
\eta_N^m(\mu) := \left( \frac{\Delta t}{\beta_\mathrm{LB}^2(\mu)} \sum_{m'=1}^m \|R_N^{m'}(\mu)\|_{V}^2 \right)^{1/2}.
\end{equation}

\begin{proposition}[Residual-based a posteriori error estimator]
For any $\mu \in \mathcal{D}$, the reduced basis approximation satisfies
\begin{equation}
\|u_h(\mu)-u_N(\mu)\|_{\mathcal{X}_h} \le \eta_N^M(\mu) \le C_e(\mu)\,\|u_h(\mu)-u_N(\mu)\|_{\mathcal{X}_h},
\label{err ineq}
\end{equation}
for some constant $C_e(\mu)>0$.
\end{proposition}

\begin{proof}
    By construction, the initial condition is represented exactly in the reduced space. 
    From the discrete inf--sup condition, we obtain
    \begin{align}
      \|u_h(\mu)-u_N(\mu)\|_{\mathcal{X}_h} &\leq \sup_{v_h \in \mathcal{Y}_h} \frac{|b(u_h(\mu)-u_N(\mu), v_h;\mu)|}{\beta_h(\mu) \|v_h\|_{\mathcal{Y}}} \\
      &=\sup_{v _h \in \mathcal{Y}_h}\frac{| \sum_{m=1}^M \Delta t \ r_N^m(v_h^m; \mu) |}{\beta_h(\mu) \left(\sum_{m=1}^M  \Delta t\|v_h^m\|_V^2 + \|v_h^0\|_H^2 \right)^{1/2}}\\
      & \leq \sup_{v_h \in \mathcal{Y}_h}\frac{ |\sum_{m=1}^M \Delta t \ r_N^m(v_h^m; \mu)| }{\beta_h(\mu) \left(\sum_{m=1}^M  \Delta t\|v_h^m\|_V^2 \right)^{1/2}}\\
      &\leq \sup_{v_h \in \mathcal{Y}_h}\frac{ \sum_{m=1}^M \Delta t \| r_N^m\|_{V_h'} \|v_h^m\|_V }{\beta_h(\mu) \left(\sum_{m=1}^M  \Delta t\|v_h^m\|_V^2 \right)^{1/2}}.
    \end{align}
    We apply Cauchy–Schwarz's inequality to get
    \begin{align}
      \|u_h(\mu)-u_N(\mu)\|_{\mathcal{X}_h} &\leq \sup_{v_h \in \mathcal{Y}_h} \frac{\left( \sum_{m=1}^M \Delta t \| r_N^m\|_{V_h'}^2 \right)^{1/2}  \left( \sum_{m=1}^M \Delta t \|v_h^m\|_V^2 \right)^{1/2} }{\beta_h(\mu) \left(\sum_{m=1}^M  \Delta t\|v_h^m\|_V^2 \right)^{1/2}}\\
      &\leq \left( \frac{\Delta t}{\beta_\mathrm{LB}^2(\mu)} \sum_{m=1}^M   \| r_N^m\|_{V_h'}^2 \right)^{1/2},
      \end{align}
      which proves the first inequality in \eqref{err ineq}.
      
     For the second inequality, we begin with the error equation \eqref{res error}. Applying the Cauchy--Schwarz inequality together with the continuity of the bilinear form $a(\cdot,\cdot;\mu)$, to obtain
    \begin{equation}
        \|r_N^m(v;\mu)\|_{V_h'} \le  \|\delta  (u_h^m(\mu)-u_N^m(\mu))\|_{V_h'} +\gamma_h(\mu)\| u_h^m(\mu)-u_N^m(\mu)\|_V.
    \end{equation}
    Hence,
    \begin{equation}
        \|r_N^m(\mu)\|_{V_h'}^2 \le 2\max\{1,\gamma_h(\mu)^2\} \left( \|\delta  (u_h^m(\mu)-u_N^m(\mu))\|_{V_h'}^2 + \| u_h^m(\mu)-u_N^m(\mu)\|_V^2\right).
    \end{equation}
    Summing over $m$ and multiplying by $\Delta t/\beta_\mathrm{LB}^2(\mu)$ gives
    \begin{equation}
        \eta_N^M(\mu)^2
        \le \frac{2}{\beta_\mathrm{LB}^2(\mu)} \max\left\{1,\gamma_{\mathrm{UB}}(\mu)^2\right\} \| u_h(\mu)-u_N(\mu)\|_{\mathcal{X}_h}^2.
    \end{equation}
    Taking square roots concludes the proof with
    \begin{equation}
        C_e(\mu) = \sqrt{\frac{2}{\beta_\mathrm{LB}^2(\mu)} \max\{1,\gamma_{\mathrm{UB}}(\mu)^2\}},
    \end{equation}
    where 
    \begin{equation}
        C_e(\mu) \leq \overline{C_e}:= \sqrt{\frac{2}{\underline{\beta}} \max\{1,\overline{\gamma}^2\}},
    \end{equation}
    
\end{proof}
In general, a computable analytical expression for the lower bound $\beta_{\mathrm{LB}}(\mu)$ of the inf--sup constant is not available and must be obtained numerically. An efficient offline--online procedure for the construction and evaluation of $\beta_{\mathrm{LB}}(\mu)$ is provided by the Successive Constraint Method (SCM); see \cite{huynh2007successive,rozza2008reduced}.

\subsection{Localized RB method}
We extend the adaptive domain decomposition algorithm proposed in \cite{barakat2026convergence}, originally developed for elliptic problems, to the setting of parametric parabolic PDEs. The main modification concerns the construction of the local reduced spaces: instead of using a purely greedy strategy, we employ the POD/Greedy approach discussed above.

The algorithm builds an adaptive partition of the parameter domain $\mathcal{D}$ by means of a binary tree. At each level $l$, the nodes are indexed by Boolean vectors
\begin{equation}
\mathcal{B}_l := {1}\times\{0,1\}^{l-1}.
\end{equation}
Each vector $B_l \in \mathcal{B}_l$ uniquely identifies a node of the tree and is associated with a parameter subdomain $\mathcal{D}_{B_l}\subset\mathcal{D}$. Appending a $0$ or $1$ to $B_l$ corresponds to moving to the left or right child node, respectively.

To initialize Algorithm~\ref{alg:nonlinear}, the localized RB algorithm, we specify the number of parameters $N$, the tolerance $\epsilon$, and the initial training sample $E_{B_1}$. For each current subdomain $\mathcal{D}_{B_l}$, a finite training set $E_{B_l}$ is defined. A local reduced basis space is then constructed using the POD/Greedy procedure described in Algorithm~\ref{alg:PODGreedy}. The number of POD modes retained at each step is chosen as
\begin{equation}
i_n:=\min \left\{ j \in \{1,\dots,M\} \;\middle|\; \sqrt{M\sum_{m=j+1 }^{M} \lambda^{m}(\mu^n)} \leq C_{\mathrm{opt}}\, \ell_{B_l}^{N^{1/d}}, \right\},
\end{equation}
where $\lambda^{m}(\mu^n)$ denotes the POD eigenvalues associated with the trajectory
$e^m(\mu^{(n)})$, for $1 \le m \le M$, ordered in decreasing order, and $\ell_{B_l}$ is the length of the longest edge of the subdomain $\mathcal{D}_{B_l}$. Further details on this choice of $i_n$ and on the value of the constant $C_{\mathrm{opt}}$ are provided in Section~\ref{pract}.

Once the reduced space has been constructed, a local error estimator $\eta_{B_l}$ is evaluated over the training set associated with the current subdomain. If the maximum error estimator over all the parameters in the training sample $E_{B_l}$ satisfies the prescribed tolerance $\epsilon$, the subdomain is accepted and no further refinement is performed. Otherwise, the subdomain is bisected into two equal parts while preserving the tensor-product structure. When $d > 1$, the splitting direction is chosen along the longest edge of the subdomain in order to maintain a balanced partition. While this choice is essential for the theoretical analysis, other splitting strategies may be used in practice. The node $B_l$ is then replaced by its two child nodes, denoted by $(B_l,0)$ and $(B_l,1)$. This procedure is applied recursively until the prescribed tolerance is satisfied on every subdomain. The final output is a collection of $K$ leaf nodes of the binary tree, each associated with a parameter subdomain and its corresponding reduced basis space.

During the online stage, for a given parameter $\mu\in\mathcal{D}$, the algorithm identifies the subdomain containing $\mu$ by comparing its coordinates with the bounding points of each subdomain in the partition. The reduced model associated with the corresponding leaf node is then used to compute the approximation.

Let $K_l$, $1 \leq l \leq L$, denote the number of subdomains at the $l$-th iteration of the approximation algorithm. This quantity corresponds to the total number of nodes at level $l$ of the tree together with the leaves generated at all previous levels. For convenience of notation, when $l$ is fixed, we replace the Boolean vector used to label each subdomain at iteration $l$ by a scalar index $k$, $1 \leq k \leq K_l$. Given a parameter $\mu \in \mathcal{D}$, we determine the subdomain $\mathcal{D}_k \subset \mathcal{D}$ that contains $\mu$. The corresponding local RB approximation, denoted by $u_{N,k}(\mu)$, is obtained by solving the reduced problem
\begin{equation}
    b(u_{N,k}(\mu), v_N;\mu)
    =
    \sum_{m=1}^M \Delta t\, g(t^m) f(v_N^m;\mu)
    +
    ( u_0(\mu), v_N^0)_H
    \quad
    \forall v_N\in\mathcal Y_{N,k},
    \label{localRB}
\end{equation}
where  $\mathcal{X}_{N,k}$ is the reduced basis space associated with the subdomain $\mathcal{D}_k$, and $\mathcal Y_{N,k}=\mathcal{X}_{N,k}$ is the associated test space.

The corresponding local reduced basis error estimator is defined by
\begin{equation}
\eta_{N,k}^m(\mu) := \left( \frac{\Delta t}{\beta_\mathrm{LB}^2(\mu)} \sum_{m'=1}^m \|r_{N,k}^{m'}(\mu)\|_{V_h'}^2 \right)^{1/2},
\label{local est}
\end{equation}
where $r_{N,k}^{m'}(\mu)$ denotes the residual associated to problem \eqref{localRB} at time step $m'$.

\begin{algorithm}[H]
\caption{LocalRB$(\epsilon, N, E_{B_l})$}
\label{alg:nonlinear}
    \begin{algorithmic}[1]
    \State Select $\mu_{B_l}^{(1)}$ randomly from $E_{B_l}$.
    \State $[V_{B_l},\epsilon_{B_l}]= \mathrm{POD/Greedy }(N,\epsilon,\mu_{B_l}^{(1)},E_{B_l})$
    \If {$\epsilon_{B_l} \leq \epsilon$}
    \State break
    \Else
    \State split the subdomain in half along the longest edge.
    \State construct $E_{(B_l,0)}$ and $E_{(B_l,1)}$.
    \State LocalRB $(\epsilon,N, E_{(B_l,0)})$.
    \State LocalRB $(\epsilon,N, E_{(B_l,1)})$.
    \EndIf
    \end{algorithmic}
\end{algorithm}

\section{Convergence analysis}
We now consider the spaces $V,H$ and $V'$ as spaces of complex-valued functions, and extend the  variational formulations to such spaces in a straightforward manner. The following assumption concerns the parameter dependence of the variational problem and its associated data.
\begin{assumption}
    The affine representations \eqref{aff 1}--\eqref{i.c aff} admit extensions to an open set
    $\mathcal O\subset\mathbb C^d$ containing the parameter domain $\mathcal D$. The extended maps $z \mapsto a(\cdot,\cdot;z)$, $z \mapsto f(\cdot;z)$, and $z \mapsto u_0(z)$ are holomorphic in each variable $z_j$, $j=1,\dots,d$.
    Moreover, for every $z \in \mathcal{O}$, the form $a(\cdot,\cdot;z)$ is uniformly continuous and coercive, i.e., there exist constants $\overline{\gamma}, \underline{\alpha} > 0$, independent of $z$, such that
    \begin{equation}
    |a(v,w;z)| \le \overline{\gamma}\,\|v\|_V \|w\|_V,
    \qquad
    \Re\, a(v,v;z) \ge \underline{\alpha}\,\|v\|_V^2
    \quad \forall v,w \in V.
    \end{equation}
    \label{assum map}
\end{assumption}
\noindent By Corollary~\ref{uni bound}, the fully discrete solution satisfies the uniform estimate
\begin{equation}
    \sup_{z \in \mathcal{O}}\|u_h(z)\|_{\mathcal{X}_h}^2
    \leq C\sup_{z \in \mathcal{O}} \Big(\|u_h^0(z)\|_H^2 + \|g\|_{L^2(0,T)}^2\|f(z)\|_{V_h'}^2 \Big),
    \label{uni bound complex}
\end{equation}
where the constant $C>0$ is independent of $z \in \mathcal{O}$.

Now we analyze the parameter-to-solution map $z \mapsto u_h^m(z)$ and parameter-to-discrete-time derivative map $z \mapsto \delta u_h^m(z)$. For each time step, the discrete problem \eqref{eq:discrete} can be written as
\begin{equation}
    B(u_h^m(z),v_h;z)=L_m(v_h;z),  \quad \forall v_h \in V_h, \quad 1\leq m \leq M,
\end{equation}
where
\begin{equation}
    B(u_h^m(z),v_h;z):=\frac{1}{\Delta t}\left( u_h^m(z), v_h \right)_H + a \left( u_h^m(z), v_h; z \right),
\end{equation}
and
\begin{equation}
    L_m(v_h;z):=\frac{1}{\Delta t}\left( u_h^{m-1}(z), v_h \right)_H+ g(t^m) f(v_h;z).
\end{equation}
Since $a(\cdot,\cdot;z)$ is continuous and coercive on $V_h$, the bilinear form $B(\cdot,\cdot;z)$ is also continuous and coercive. Hence, by the Lax--Milgram theorem, the bounded linear operator $\mathcal{B}(z):V_h\to V_h'$ induced by $B(\cdot,\cdot;z)$ is an isomorphism for every $z\in\mathcal O$. Therefore, for each time step $m$,
\begin{equation}
    u_h^m(z)=\mathcal{B}(z)^{-1}L_m(z).
\end{equation}
By Assumption~\ref{assum map}, the map $z\mapsto \mathcal{B}(z)$ is holomorphic from $\mathcal O$ into $\mathcal L(V_h,V_h')$. Since $\mathcal{B}(z)$ is invertible for every $z\in\mathcal O$, the map $z\mapsto \mathcal{B}(z)^{-1}$ is also holomorphic on $\mathcal O$.

For $m=1$, the functional $L_1(z)$ inherits holomorphic dependence from $u_h^0(z)$ and $f(\cdot;z)$. Hence, the map
\begin{equation}
    \mu \mapsto u_h^1(z)=\mathcal{B}(z)^{-1}L_1(z)
\end{equation}
is holomorphic in $\mathcal O$. Now assume that $u_h^{m-1}(z)$ has holomorphic dependence on $z$. Since $L_m(z)$ depends affinely on $u_h^{m-1}(z)$ and $f(\cdot;z)$, the map
\begin{equation}
    \mu \mapsto u_h^m(z)=\mathcal{B}(z)^{-1}L_m(z)
\end{equation}
is holomorphic in $\mathcal O$. By induction, the parameter-to-solution map $z\mapsto u_h^m(z)$ is holomorphic in $\mathcal O$ for every $m=1,\ldots,M$. Finally, the map $z \mapsto \delta u_h^m(z)$ is also holomorphic in $\mathcal{O}$, being a linear combination of the maps
$z\mapsto u_h^m(z)$ and $z\mapsto u_h^{m-1}(z)$.

We suppose that the parameter domain $\mathcal{D}$ is decomposed into $K$ subdomains, each characterized by side lengths 
\begin{equation} 
\ell_k = (\ell_{k,1},\ell_{k,2},\dots,\ell_{k,d}),  \qquad 1 \le k \le K.
\end{equation}
Let $\mathcal{D}_k \subset \mathcal{D}$ denote a generic subdomain centered at $\hat{\mu}_k \in \mathcal{D}$. By construction, we have $\mathcal{D}_k \subset \mathcal{O}$.
For a suitable vector $\rho_k = (\rho_{k,j})_{j=1}^d$, we introduce the associated polydisc
\begin{equation}
\mathcal{P}_{\rho_k} := \left\{ z=(z_j)_{j=1}^d \;:\; |z_j-\hat{\mu}_{k,j}| \le \rho_{k,j} \right\} = \bigotimes_{j=1}^d \left\{ |z_j-\hat{\mu}_{k,j}| \le \rho_{k,j} \right\}.
\end{equation}
The radii $\rho_k$ are chosen so that $\mathcal{D}_k \subset \mathcal{P}_{\rho_k} \subset \mathcal{O}$. One admissible choice is
\begin{equation}
\rho_{k,j} = \frac{\ell_{k,j}}{2} + \varepsilon, \qquad 1 \le j \le d,
\end{equation}
for some $\varepsilon>0$. For $\mu \in \mathcal{D}_k$, where
\begin{equation}
\mathcal{D}_k = \prod_{j=1}^d I_{k,j},
\qquad
I_{k,j} = \left[\hat{\mu}_{k,j} - \frac{\ell_{k,j}}{2}, \hat{\mu}_{k,j} + \frac{\ell_{k,j}}{2} \right],
\end{equation}
we define the normalized variables
\begin{equation}
\mu_{k,j} = \frac{2(\mu_j-\hat{\mu}_{k,j})}{\ell_{k,j}} \in [-1,1].
\label{para norm}
\end{equation}
This transformation naturally extends to the polydisc $\mathcal{P}_{\rho_k}$ and to the open set $\mathcal{O}$. Under this scaling, $\mathcal{P}_{\rho_k}$ is mapped to a rescaled polydisc $\mathcal{P}_{\hat{\rho}_k}$ with radii
\begin{equation}
\hat{\rho}_{k,j} = 1 + \frac{2\varepsilon}{\ell_{k,j}}, \qquad 1 \le j \le d,
\end{equation}
and $\mathcal{O}$ is transformed into its normalized and shifted counterpart, denoted by $\hat{\mathcal{O}}$.
Accordingly, for every $\mu \in \mathcal{D}_k$, the discrete solution admits the representation
\begin{equation}
u_h^m(\mu) = u_{h,k}^m(\mu_k),
\end{equation}
where $u_{h,k}^m(\mu_k)$ denotes the solution of the discrete parametric problem corresponding to the normalized parameterization defined by \eqref{para norm}.

For any given subdomain $\mathcal{D}_k \subset \mathcal{D}$, $1 \leq k \leq K$, we construct a local polynomial approximation based on the multivariate Taylor expansion of the normalized solution. For $1 \le m \le M$, we write
\begin{equation}
    u_h^m(\mu)=u_{h,k}^m(\mu_k)=\sum_{\nu \in \mathbb{N}^d} t_{k,\nu}^m (\mu_k)^{\nu}, 
\end{equation}
where the Taylor coefficients are defined by
\begin{equation}
    t_{k,\nu}^m=\frac{\partial^{\nu}u_{h,k}^m(0)}{\nu!}
\end{equation}
and the standard multi-index notation
\begin{equation}
    (\mu_k)^{\nu}=\prod_{j=1}^d (\mu_{k,j})^{\nu_j}
\end{equation}
is employed. The discrete-time derivative of the Taylor coefficients satisfies 
\begin{equation}
    \delta t_{k,\nu}^m=\frac{t_{k,\nu}^m-t_{k,\nu}^{m-1}}{\Delta t}=\frac{1}{\Delta t}\left(\frac{\partial^{\nu}u_{h,k}^m(0)}{\nu!}-\frac{\partial^{\nu}u_{h,k}^{m-1}(0)}{\nu!}\right)=\frac{\partial^{\nu}\delta u_{h,k}^m(0)}{\nu!}.
\end{equation}

\begin{lemma}[Bound on the $\mathcal{X}_h$-norm of the Taylor coefficients]
    Consider a discrete parabolic parametric PDE of the form \eqref{eq:discrete} with a $d$-dimensional parameter domain satisfying Assumption~1. Furthermore, let $\mathcal{D}$ be partitioned into $K$ tensor-product-structured subdomains ${\mathcal{D}_k}$, $1 \leq k \leq K$. Then there exists a constant $C>0$, such that
    \begin{equation}
    \left\|t_{k,\nu}\right\|_{\mathcal{X}_h} \leq C \hat\rho_k^{-\nu}=C \prod_{j=1}^d \hat\rho_{k,j}^{-\nu_j}, \quad \nu \in \mathbb{N}^d.
    \label{taylor bound}
    \end{equation}
    \label{lemma taylor}
\end{lemma}
\begin{proof}
    The map  $z \mapsto u_{h,k}(z)$ is holomorphic in the set $\hat{\mathcal{O}}$  containing the polydisc $\mathcal{P}_{\hat\rho_k}$. The Cauchy integral formula may be applied successively in each variable. For any $\tilde z$ in the interior of $\mathcal{P}_{\hat\rho_k}$, we have
    \begin{equation}
    u_{h,k}^m(\tilde{z}_1, \dots, \tilde{z}_d)
    = (2\pi i)^{-d} \int_{|z_1| = \hat\rho_{k,1}} \cdots \int_{|z_d| = \hat\rho_{k,d}} 
    \frac{u_{h,k}^m(z_1, \dots, z_d)}{(\tilde{z}_1 - z_1) \cdots (\tilde{z}_d - z_d)} \, \mathrm{d}z_1 \cdots \mathrm{d}z_d.
    \end{equation}
    Differentiating this expression yields
    \begin{equation}
    \frac{\partial^{|\nu|}}{\partial \tilde{z}_1^{\nu_1} \cdots \partial \tilde{z}_d^{\nu_d}} u_{h,k}^m(0, \dots, 0)
    = \nu! (2\pi i)^{-d} \int_{|z_1| = \hat\rho_{k,1}} \cdots \int_{|z_d| = \hat\rho_{k,d}}
    \frac{u_{h,k}^m(z_1, \dots, z_d)}{z_1^{\nu_1 + 1} \cdots z_d^{\nu_d + 1}} \, \mathrm{d}z_1 \cdots \mathrm{d}z_d.
    \end{equation}
    Then,
    \begin{equation}
    \left\| \partial^{\nu} u_{h,k}^m(0) \right\|_V = \left\| \frac{\partial^{|\nu|} u_{h,k}^m}{\partial \tilde{z}_1^{\nu_1} \cdots \partial \tilde{z}_d^{\nu_d}}(0, \dots, 0) \right\|_V 
    \leq \nu! \|u_{h,k}^m(z_1, \dots, z_d)\|_V  \prod_{j \leq d} \hat\rho_{k,j}^{-\nu_j}.
    \end{equation}
    and
    \begin{equation}
    \left\| \partial^{\nu} u_{h,k}^M(0) \right\|_H  
    \leq \nu! \|u_{h,k}^M(z_1, \dots, z_d)\|_H \prod_{j \leq d} \hat\rho_{k,j}^{-\nu_j}.
    \end{equation}
    Similarly, 
    \begin{equation}
    \left\| \partial^{\nu} \delta u_{h,k}^m(0) \right\|_{V_h'}
    \leq \nu! \|\delta u_{h,k}^m(z_1, \dots, z_d)\|_{V_h'}  \prod_{j \leq d} \hat\rho_{k,j}^{-\nu_j},
    \end{equation}
    \begin{equation}
    \left\| \partial^{\nu} \delta u_{h,k}^m(0) \right\|_{H}
    \leq \nu! \|\delta u_{h,k}^m(z_1, \dots, z_d)\|_{H}  \prod_{j \leq d} \hat\rho_{k,j}^{-\nu_j}.
    \end{equation}
    Therefore,
    \begin{align}
        \left\|t_{k,\nu}\right\|_{\mathcal{X}_h}&\leq \|u_{h,k}(z)\|_{\mathcal{X}_h} \prod_{j \leq d} \hat\rho_{k,j}^{-\nu_j}.
    \end{align}
    Using the uniform bound \eqref{uni bound complex}, extended to the set $\hat{\mathcal{O}}$, we obtain the desired estimate.
\end{proof}

\begin{proposition}
Consider the discrete parabolic parametric PDE \eqref{eq:discrete}, and assume that the hypotheses of Lemma~\ref{lemma taylor} hold. Let $\ell$ denote the maximum side length of the partition. Then there exists a constant $\hat{C}>0$, independent of $\ell$, such that
\begin{equation}
\max_{1 \leq k \leq K} d_{N}(\mathcal{M}_{h,k})
\leq \hat{C}\, \ell^{N^{1/d}},
\label{hat C}
\end{equation}
where $d_{N}(\mathcal{M}_{h,k})$ denotes the Kolmogorov $N$-width of the solution manifold associated with the $k$th subdomain.
\end{proposition}

\begin{proof}
The proof is analogous to that of Proposition~4.1 in \cite{barakat2026convergence}. The key difference is that the $V$-norm bound on the Taylor coefficients used in the elliptic analysis is replaced by the $\mathcal{X}_h$-norm bound established in Lemma~\ref{lemma taylor}.
\end{proof}

\subsection{Convergence of the localized RB method}
For any subdomain $\mathcal{D}_k$ and $\mu \in \mathcal{D}_k$, let $u_{N,k}(\mu)$ denote the solution of the reduced problem \eqref{localRB} associated with the local reduced space $V_{N,k}$ generated by Algorithm~\ref{alg:PODGreedy}, i.e.,
\begin{equation}
    [V_{N,k},\epsilon_{\max}]=\mathrm{POD/Greedy}(N,\epsilon,\mu_k^{(1)},E_k).
    \label{POD Rb space}
\end{equation}
Furthermore, let $\hat{u}_{N,k}(\mu)$ denote the solution of \eqref{localRB} obtained using the local reduced space
\begin{equation}
    \hat V_{N,k}:=\text{span}\{u_h^m(\mu_k^{(n)}):0\leq m \leq M, \quad 1\leq n \leq N  \},
    \label{full Rb space}
\end{equation}
where the entire time trajectory corresponding to each selected parameter is retained at every iteration, and no POD compression is performed.
\noindent By construction, $V_{N,k}\subseteq \hat V_{N,k}$. Consequently, the reduced approximation based on $\hat V_{N,k}$ is at least as accurate as that based on $V_{N,k}$, yielding
\begin{equation}
     \|u_{h}(\mu)-\hat u_{N,k} 
    (\mu)\|_{\mathcal{X}_h} \leq\|u_{h}(\mu)-u_{N,k}(\mu)\|_{\mathcal{X}_h} \leq  \eta_{N,k}^M(\mu),
    \label{LB}
\end{equation}
where $\eta_{N,k}^M(\mu)$ is the local estimator~\eqref{local est} associated with the local RB space $V_{N,k}$. Furthermore,
\begin{align}
    \eta_{N,k}^M(\mu) &\leq \overline{C_e} \|u_{h}(\mu)- u_{N,k}(\mu)\|_{\mathcal{X}_h}\\
    &\leq  \overline{C_e}  \{\|u_{h}(\mu)- \hat u_{N,k}(\mu)\|_{\mathcal{X}_h}+ \|\hat u_{N,k}(\mu)- u_{N,k}(\mu)\|_{\mathcal{X}_h}\}.
    \label{two errors}
\end{align}
The second term on the right-hand represents the error introduced by POD truncation and can be made arbitrarily small by retaining a sufficiently large number of POD modes. Hence, assuming that enough POD modes are retained, we have
\begin{align}
    \eta_{N,k}^M(\mu) \leq  2 \, \overline{C_e}  \|u_{h}(\mu)- \hat u_{N,k}(\mu)\|_{\mathcal{X}_h}.
    \label{est}
\end{align}
Applying the quasi-optimality estimate \eqref{quais_X_h} to $u_{h}(\mu)$ and $\hat u_{N,k}(\mu)$ and invoking the weak greedy estimate \eqref{th comp}, we obtain
\begin{equation}
    \|u_{h}(\mu)-\hat u_{N,k}(\mu)\|_{\mathcal{X}_h} \leq C_h \|u_{h}(\mu)- P_N u_{h,k}(\mu)\|_{\mathcal{X}_h} \leq C_h\kappa d_{N}(\mathcal{M}_{h,k}).
    \label{err1}
\end{equation}
Therefore,
\begin{equation}
    \eta_{N,k}^M(\mu) \leq 2 \, \overline{C_e}C_h\kappa d_{N}(\mathcal{M}_{h,k}).
    \label{upper est}
\end{equation}

Finally, we say that the parameter domain partition resulting from Algorithm~\ref{alg:nonlinear} is quasi-uniform if, for every tolerance $\epsilon>0$, the sizes of all subdomains are comparable, in the sense that the subdomain with the smallest volume is bounded below by a fixed proportion of the largest subdomain (see Assumption 2 in \cite{barakat2026convergence}).

\begin{theorem}
    Consider a discrete parabolic parametric PDE of the form \eqref{eq:discrete} with a $d$-dimensional parameter domain such that Assumption~\ref{assum map} is satisfied. Let $N$ denote the number of parameters used by Algorithm~\ref{alg:nonlinear} in the construction of the localized RB approximation. For any prescribed tolerance $\epsilon>0$, let $K(\epsilon)$ denote the number of parameter subdomains generated by Algorithm~\ref{alg:nonlinear}. If the parameter domain partition is quasi-uniform and sufficiently many POD modes are retained so that estimate \eqref{est} holds at each iteration, then there exists a constant $C>0$, independent of $\epsilon$, such that
    \begin{equation}
    K(\epsilon) \leq \max\left\{1,\frac{C}{\epsilon^{\,d/N^{1/d}}}\right\}.
    \end{equation}
    \label{th conv}
\end{theorem}

\begin{proof}
     Algorithm~\ref{alg:nonlinear} terminates once the error estimator on each subdomain satisfies the prescribed tolerance $\epsilon$. If $L=1$ (and hence $K=1$), the proof is complete. Otherwise, assume that $L>1$. For $1 \le l \le L-1$ and $1 \le k \le K_l$, we denote by
    \begin{equation}
    \hat{\eta}_l := \max \{ \eta_{N,k}^M(\mu) : \mu \in \mathcal{D}_k,\; 1 \le k \le K_l \}
    \end{equation}
    the maximum a posteriori error estimator at the $l$-th iteration. By the stopping criterion of the algorithm, it follows that
    \begin{equation}
    \epsilon < \hat{\eta}_l .
    \label{eps=1}
    \end{equation}
    
    Next, we define
    \begin{equation}
    \ell_l := \max_{\substack{1 \le k \le K_l \\ 1 \le j \le d}} \ell_{k,j},
    \label{ell}
    \end{equation}
    which represents the maximum side length of the parameter subdomains at the $l$-th iteration. Using the bounds \eqref{upper est} and \eqref{hat C}, we obtain
    \begin{align}
    \hat{\eta}_l 
    &\leq 2 \, \overline{C_e}C_h \hat{C} \kappa \ell_l^{\,N^{1/d}} .
    \label{eta_1}
    \end{align}
    
    The remainder of the proof then follows by the same arguments as in the elliptic case; see Theorem 4.1 in \cite{barakat2026convergence}.
\end{proof}

\subsection{Practical takeaways}
\label{pract}
For any given parameter value, the RB approximation error arises from two distinct sources. The first is due to the reduction of the solution space through the greedy algorithm, while the second stems from the POD compression of the time-dependent solution trajectories. These two contributions are explicitly identified in~\eqref{two errors}. When the error associated with the POD compression is dominated by the error due to the reduction of the solution space, the estimator exhibits the asymptotic decay rate $\mathcal{O}\!\left(\ell^{N^{1/d}}\right)$. In the following, we show that the error introduced by the temporal trajectory compression can be bounded by the corresponding error evaluated at the parameter samples selected by the greedy algorithm during the construction of the reduced basis space.

At the $l$-th iteration of Algorithm~\ref{alg:nonlinear}, for each subdomain $\mathcal{D}_k$, $1 \leq k \leq K_l$, we assume that the basis functions spanning the reduced basis space $\hat V_{k,N}$ defined in \eqref{full Rb space} are orthonormal with respect to the inner product $(\cdot,\cdot)_V$. In particular,
\begin{equation}
    \hat V_{k,N} = \text{span} \{\phi_k^m(\mu^{(n)}) : 0 \leq m \leq M,\; 
    1 \leq n \leq N \}.
\end{equation}
with
\begin{equation}
    (\phi_k^m(\mu^{(n)}),\phi_k^{m'}(\mu^{(n')}))_V
=
\delta_{mm'}\delta_{nn'}.
\end{equation}
Accordingly, for any given parameter $\mu \in \mathcal{D}$, the reduced solution $\hat{u}_{N,k}(\mu)$ admits the expansion
\begin{equation}
    \hat{u}_{N,k}^j(\mu) = \sum_{n=1}^{N} \sum_{m=0}^{M} 
    c_{k,j}^{m,n}\,\phi_k^m(\mu^{(n)}), 
    \qquad 1 \leq j \leq M.
    \label{expansion}
\end{equation}
In what follows, $C > 0$ denotes a generic constant whose value may change from one occurrence to the next. Since the basis functions are orthonormal, estimate~\eqref{RB uni bound} implies
\begin{equation}
    \sum_{j=1}^{M} \Delta t \sum_{n=1}^{N} \sum_{m=1}^{M} 
     |c_{k,j}^{m,n} |^2 \leq C.
     \label{coeff bound}
\end{equation}

Let ${u}_{N,k}$ denotes the RB solution based on the POD generated space $V_{k,N}$ defined in \eqref{POD Rb space}, and define the error at each time step by $\textrm{err}^j := \hat{u}_{N,k}^j(\mu) - u_{N,k}^j(\mu)$. 
Expanding the graph norm yields
\begin{equation}
    \|\textrm{err}r\|_{\mathcal{X}_h}^2 
    = \sum_{j=1}^{M} \Delta t\,\|\textrm{err}^j\|_V^2  +  \sum_{j=1}^{M} \Delta t \|\delta \textrm{err}^j\|_{V_h'}^2 +   \sum_{j=1}^{M} \Delta t^2\|\delta \textrm{err}^j\|_{H}^2 + \|\textrm{err}^M\|_H^2.
\end{equation}
Using the definition of the discrete time derivative, namely, $\delta \textrm{err}^j= (\textrm{err}^j - \textrm{err}^{j-1})/\Delta t$, together with the inequality $\|a-b\|^2 \le 2\left(\|a\|^2+\|b\|^2\right),$ we obtain
\begin{equation}
\begin{split}
    \|\textrm{err}r\|_{\mathcal{X}_h}^2 
    &\leq \sum_{j=1}^{M} \Delta t \|\textrm{err}^j\|_V^2 
       + \sum_{j=1}^{M} \frac{2}{\Delta t} \left(\|\textrm{err}^j\|_{V_h'}^2 + \|\textrm{err}^{j-1}\|_{V_h'}^2 \right)\\
        &+\sum_{j=1}^{M} 2\left(\|\textrm{err}^j\|_{H}^2 + \|\textrm{err}^{j-1}\|_{H}^2\right) +\|\textrm{err}^M\|_H^2.
\end{split}
\end{equation}
Invoking the embedding inequalities $\|v\|_{H}\leq C_P \|v\|_{V}$ and $ \|v\|_{V_h'} \leq C_P^2 \|v\|_{V}$ gives
\begin{align}
    \|\textrm{err}\|_{\mathcal{X}_h}^2 
    & \leq \sum_{j=1}^{M} \Delta t\left(\|\textrm{err}^j\|_V^2 
       + \frac{4C_p^4}{\Delta t^2}\|\textrm{err}^j\|_V^2+ \frac{4C_p^2}{\Delta t}\|\textrm{err}^j\|_V^2+ \frac{C_p^2}{\Delta t}\|\textrm{err}^j\|_V^2\right)\\
    &\leq \max\!\left\{1,\,\frac{4C_p^4}{\Delta t^2} ,\, \frac{4C_p^2}{\Delta t} \right\}
       \sum_{j=1}^{M} \Delta t\,\|\textrm{err}^j\|_V^2.
\end{align}
Applying Céa's lemma at each time step $j$, and absorbing the constants, we obtain
\begin{equation}
     \|\hat{u}_{N,k}(\mu) - u_{N,k}(\mu) \|_{\mathcal{X}_h}^2
    \leq \frac{C}{\Delta t^2}        \sum_{j=1}^{M} \Delta t
     \|\hat{u}_{N,k}^j(\mu) - P_{ V _{k,N}} \hat{u}_{N,k}^j(\mu)\|_V^2,
\end{equation}
where $P_{ V _{k,N}}$ denotes the $V$--orthogonal projection onto the space $V _{k,N}$. Using the expansion \eqref{expansion}, we get
\begin{equation}
    \|\hat{u}_{N,k}(\mu) - u_{N,k}(\mu)\|_{\mathcal{X}_h}^2
    \leq \frac{C}{\Delta t^2}\sum_{j=1}^{M} \Delta t\,
       \left\|\sum_{n=1}^{N} \sum_{m=1}^{M} c_{k,j}^{m,n}
        \left(\phi_k^m(\mu^{(n)}) - P_{ V _{k,N}}\phi_k^m(\mu^{(n)}) \right)
       \right\|_V^2\end{equation}
Applying the Cauchy--Schwarz inequality, we deduce
\begin{equation}
     \|\hat{u}_{N,k}(\mu) - u_{N,k}(\mu)\|_{\mathcal{X}_h}^2
    \leq \frac{C}{\Delta t^2}\left(\sum_{j=1}^{M} \Delta t \sum_{n=1}^{N} \sum_{m=1}^{M}
        |c_{k,j}^{m,n} |^2\right)
       \sum_{n=1}^{N} \sum_{m=1}^{M}
        \|\phi_k^m(\mu^{(n)}) - P_{ V _{k,N}}\phi_k^m(\mu^{(n)}) \|_V^2.
\end{equation}
Recalling the uniform bound \eqref{coeff bound} on the coefficients, the final estimate reduces to
\begin{equation}
     \|\hat{u}_{N,k}(\mu) - u_{N,k}(\mu) \|_{\mathcal{X}_h}
    \leq \frac{C}{\Delta t}\sum_{n=1}^{N}\sqrt{\sum_{m=1}^{M}
        \|\phi_k^m(\mu^{(n)}) - P_{ V _{k,N}}\phi_k^m(\mu^{(n)}) \|_V^2}.
\end{equation}
Therefore, it is sufficient to control the POD compression error at the parameter values selected by the greedy algorithm in order to control the POD error uniformly over the entire parameter domain. 

Imposing the condition
\begin{equation}
    \sqrt{\sum_{m=1}^{M}
     \|\phi_k^m(\mu^{(n)}) - P_{ V _{k,N}}\phi_k^m(\mu^{(n)}) \|_V^2}
    \leq  C\,\Delta t\, \ell_l^{N^{1/d}}, \quad n=1,\ldots,N,
    \label{POD cond}
\end{equation}
ensures that
\begin{equation}
     \|\hat{u}_{N,k}(\mu) - u_{N,k}(\mu)\|_{\mathcal{X}_h} 
     \leq NC\,\ell_l^{N^{1/d}} \quad \forall \mu\in\mathcal D_k.
\end{equation}
where $\ell_l$ is defined in \eqref{ell}. Consequently, the POD compression error is of the same asymptotic order as the greedy approximation error, and estimate~\eqref{est} follows.

At iteration $n$, $1 \leq n \leq N$, of Algorithm~\ref{alg:PODGreedy}, the RB space is enriched with POD modes of the projection error
\[
e_k^m(\mu^{(n)}) = \phi_k^m(\mu^{(n)}) - P_{ V _{k,n-1}}\phi_k^m(\mu^{(n)}), 
\qquad 1 \leq j \leq M.
\]
The POD truncation satisfies
\begin{equation}
    \sqrt{\sum_{m=1}^{M}
     \|e_k^m(\mu^{(n)}) - P_{i_n}e_k^m(\mu^{(n)}) \|_V^2}
    = \sqrt{M\sum_{m=i_n +1}^{M} \lambda_k^m(\mu^{(n)})},
\end{equation}
where $P_{i_n}$ denotes the orthogonal projection onto the space spanned by the first $i_n$ POD modes of $\{ e_k^m(\mu^{(n)}) \}_{m=1}^M$.
Using the definition of $e_k^m(\mu^{(n)})$ and orthonormality, we obtain
\begin{equation}
    \sqrt{\sum_{m=1}^{M}
     \|e_k^m(\mu^{(n)}) - P_{i_n}e_k^m(\mu^{(n)}) \|_V^2}
    =
    \sqrt{\sum_{m=1}^{M}
     \|\phi_k^m(\mu^{(n)}) - P_{ V _{k,n}}\phi_k^m(\mu^{(n)}) \|_V^2}.
\end{equation}
Thus, the decay of the eigenvalues $\{\lambda_k^m(\mu^{(n)})\}$ determines the number of POD modes required at each iteration to satisfy condition~\eqref{POD cond}. Specifically, $i_n$ is chosen such that
\begin{equation}
    \sqrt{M\sum_{m=i_n +1}^{M} \lambda_k^m(\mu^{(n)})} 
    \leq C_{\mathrm{opt}}\, \Delta t \, \hat\ell_k^{N^{1/d}},
    \label{POD criterion}
\end{equation}
for some constant $C_{\mathrm{opt}}$ depending on $\Delta t$, $N$, and $d$, where $(\hat\ell_k \leq \ell_l)$ is the longest side of the $k$--th parameter subdomain. The quasi-uniformity assumption is required for the theoretical result in Theorem~\ref{th conv}; however, it is not enforced in practice, as doing so would unnecessarily increase the total number of RB spaces and thereby compromise the efficiency of the online stage. Consequently, using the local quantity $\hat\ell_k$, rather than the global maximum side length $\ell_l$, is crucial for the termination of Algorithm~\ref{alg:nonlinear}. Without enforcing quasi-uniformity, $\ell_l$ may remain unchanged during the algorithm if a subdomain attaining the maximum size already satisfies the prescribed tolerance and is therefore no longer refined. If $\ell_l$ were used in the \eqref{POD criterion}, the contribution of the POD truncation error would then remain fixed for the remaining subdomains. As a result, these subdomains may never satisfy the prescribed tolerance and could be partitioned indefinitely. The use of the local length scale $\hat\ell_k$ avoids this issue by adapting the POD tolerance to the size of each individual parameter subdomain, thereby ensuring that the refinement process can terminate without imposing unnecessary global refinement.

For the sake of online efficiency, the constant $C_{\mathrm{opt}}$ must be selected with care. If chosen too small, an excessive number of POD modes is retained, resulting in a large reduced basis space and, consequently, increased online computational cost. On the other hand, if $C_{\mathrm{opt}}$ is taken too large, the reduced basis dimension remains small due to the limited inclusion of POD modes, which may benefit the online stage. However, this comes at the expense of a significantly more expensive offline stage, as well as increased storage requirements. In particular, a finer partition of the parameter domain is then necessary to meet the prescribed tolerance, since the approximation error remains relatively large when the reduced basis space is insufficiently rich. Moreover, if $C_{\mathrm{opt}}$ is excessively large, the onset of the asymptotic convergence regime is delayed and may only become apparent for prohibitively small values of $\ell_k$. To balance these effects, we select $C_{\mathrm{opt}}$ such that
\begin{equation}
    C_{\mathrm{opt}}\, \Delta t \, \ell_1^{N^{1/d}} 
    = \zeta \max_{1 \leq n \leq N} 
    \sqrt{M\sum_{m=2}^{M} \lambda^m(\mu^{(n)})},
    \label{opt_final}
\end{equation}
for some chosen value $\zeta$, where $\lambda^m(\mu^{(n)})$ denote the POD eigenvalues associated with the trajectory $e^j(\mu^{(n)})$ at the original parameter domain. This choice ensures that, at the initial stage, only one POD mode is retained for each parameter, thereby promoting an efficient online stage, while keeping the tolerance sufficiently small to avoid delaying the onset of the asymptotic convergence regime.

 \section{Numerical results}
 We now present numerical experiments. The purpose is threefold. First, we validate the theoretical convergence result established in Theorem~\ref{th conv}. Second, we compare the online computational cost of the nonlinear method against that of the linear method. Third, we investigate the impact of the constant $C_{\mathrm{opt}}$ on both the online and offline stages by selecting different values of $\zeta$ in \eqref{opt_final}.

We consider the parametrized time-dependent convection--diffusion problem introduced in \cite{eftang2011hp2}. Let $\mu=(\mu_1,\mu_2)\in\mathcal{D}$ denote the parameter vector, where $\mu_1$ specifies the direction of the flow and $\mu_2$ its magnitude. Accordingly, the parameter-dependent velocity field is defined by
\begin{equation}
V(\mu)=
\begin{bmatrix}
\mu_2\cos\mu_1\
\mu_2\sin\mu_1
\end{bmatrix}^T.
\end{equation}
For a given parameter $\mu\in\mathcal{D}$, the scalar field $u(\mu)$ satisfies
\begin{equation}
\partial_t u(t;\mu) - \Delta u(t;\mu) + V(\mu)\cdot \nabla u(t;\mu) = 10 
\quad \text{in } \Omega, \quad t\in(0,T],
\end{equation}
subject to homogeneous Dirichlet boundary conditions and a zero initial condition where the physical domain is given by
$\Omega = \{(x,y) \in \mathbb{R}^2 : x^2 + y^2 < 2 \}$. We define the associated solution space $V := H_0^1(\Omega)$.

The corresponding parametrized weak formulation is: for a given $\mu\in\mathcal{D}$, find $u(t;\mu) \in L^2(0,T;V)\cap H^1(0,T;V')$ such that
\begin{equation}
\left\langle\partial_t u(t;\mu),v\right\rangle_{V'\times V}
+a\left(u(t;\mu),v;\mu\right)
=f(v),
\qquad \forall v\in V,\quad t\in(0,T],
\end{equation}
where the parameter-dependent bilinear form is
\begin{align}
a(w, v; \mu) &= \int_{\Omega} \nabla w \cdot \nabla v 
+ \int_{\Omega} (V(\mu) \cdot \nabla w)\, v \\
&= \int_{\Omega} \nabla w \cdot \nabla v 
+ \mu_2 \cos \mu_1 \int_{\Omega} \frac{\partial w}{\partial x} v 
+ \mu_2 \sin \mu_1 \int_{\Omega} \frac{\partial w}{\partial y} v,
\end{align}
and the linear functional is
\begin{equation}
f(v) = 10 \int_{\Omega} v, \quad \forall v \in V.
\end{equation}
We note that the parameter dependence appears only in the bilinear form $a(\cdot,\cdot;\mu)$, which admits an affine decomposition. Furthermore, the bilinear form $a(\cdot,\cdot;\mu)$ is coercive with a parameter-independent constant $\alpha = 1$, and its continuity constant satisfies $\gamma_h(\mu) = 1 + |\mu_2| C_P$, where $C_P$ denotes the Poincar\'e constant.

For the temporal discretization, we consider time levels $t^m = 0.01\,m$, for $0 \leq m \leq 100$. For the truth approximation, we consider a $\mathbb{P}_1$ finite element space $V_h \subset V$ of dimension $2300$, defined over a triangular mesh.

Finally, we define two parameter domains:
\begin{equation}
\mathcal{D}_{\mathrm{I}} := \{0\} \times [0, 10], 
\qquad 
\mathcal{D}_{\mathrm{II}} := [0, \pi] \times [0, 10],
\end{equation}
corresponding to the one-parameter case ($d=1$) and the two-parameter case ($d=2$), respectively.

We now present the convergence results for Algorithm~\ref{alg:nonlinear}. We begin by specifying the maximum number of parameters $N$ to be selected by the greedy algorithm, the size of the training set $E_{B_1}$, and the initial parameter $\mu_{B_1}^{(1)}$. The algorithm is executed for various values of the tolerance $\epsilon$, and we plot the number of subdomains $K$ as a function of $\epsilon$.

We begin with the one-parameter case, $\mathcal{D} = \mathcal{D}_{\mathrm{I}}$, and choose $(0,0)$ as the initial parameter. The training set consists of $10^2$ randomly sampled points. The POD optimality constant $C_{\mathrm{opt}}$ is chosen corresponding to $\zeta=5$. In Figure~\ref{comp_1d}, we plot the number of subdomains $K$ against the tolerance $\epsilon$, considering three different cases: $N = 1, 2,$ and $3$. The observed convergence rates align well with the theoretical results presented in Theorem~\ref{th conv}.

\begin{figure}[!htbp]
    \centering
    \includegraphics[width=.6\textwidth,height=.5\textwidth]{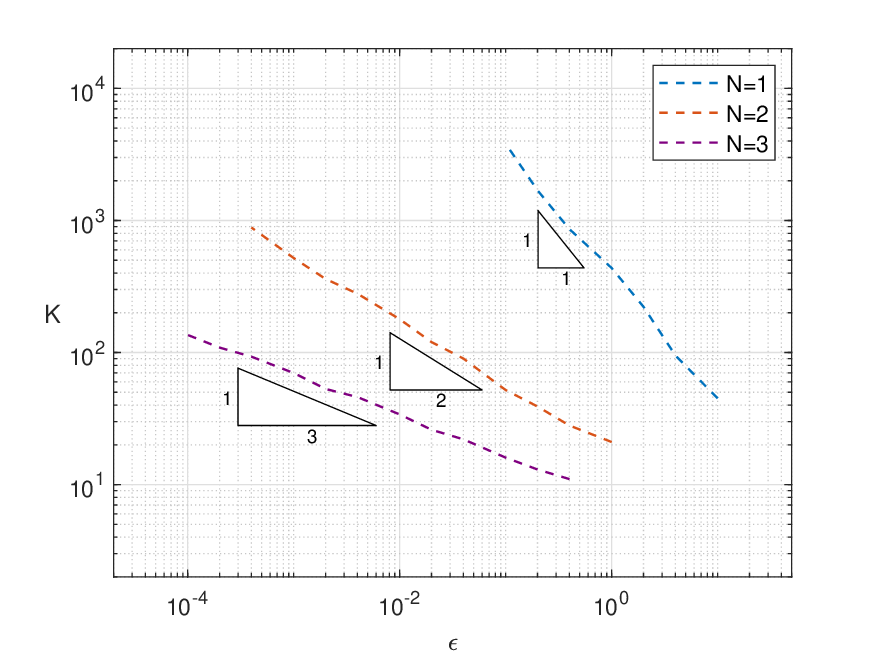}
    \caption{Number of subdomains generated by Algorithm~\ref{alg:nonlinear} to achieve a target tolerance $\epsilon$ for the case $\mathcal{D}= \mathcal{D}_{\mathrm{I}}$ with $N=1,2,$ and $3$.}
    \label{comp_1d}
\end{figure}

Next, we present the online cost improvement achieved by the nonlinear treatment. To obtain a benchmark for comparison, we construct a linear model by running Algorithm~\ref{alg:PODGreedy} with $N = \infty$, $\mu^{(1)} = (0,0)$, $\epsilon = 10^{-4}$, and a training set size of $10^3$. At each step of the algorithm, we retain only one POD mode by setting $i=1$ in Step 6. This choice is advantageous for the linear model, for which the same parameter $\mu$ may be selected multiple times. Under these settings, the tolerance is achieved with a RB space of size $53$. In Figure~\ref{linear1D}, we plot the tolerance $\epsilon$ against the maximum size $\hat{N}$ of the RB spaces in the library resulting from Algorithm~\ref{alg:nonlinear} for the two cases $N = 2$ and $N = 3$, alongside the linear model reference case (which involves no splitting of the parameter domain). As expected, the nonlinear approach improves the online cost. The case \(N=2\) yields better online performance than \(N=3\), as the tolerance is achieved with a smaller maximum RB-space size \(\hat{N}\). However, this improvement comes at a higher offline cost: the smaller value of \(N\) leads to smaller RB bases on each subdomain, requiring further partitioning to achieve the prescribed tolerance and consequently resulting in a larger library, as shown in Figure~\ref{comp_1d}. This larger library also increases the associated storage requirements.

\begin{figure}[!htbp]
    \centering
    \includegraphics[width=.6\textwidth,height=.45\textwidth]{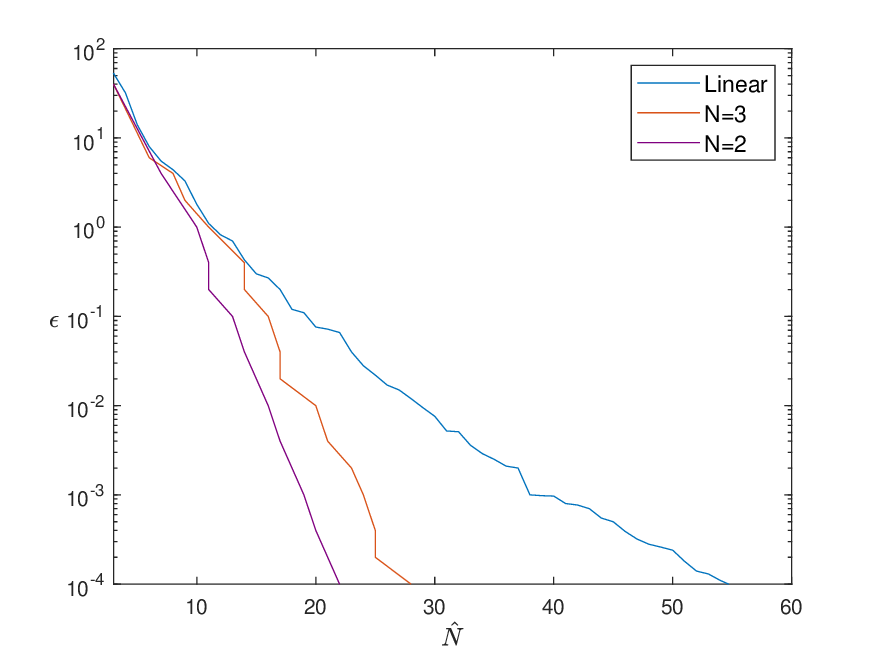}
    \caption{Convergence results for the linear model compared to the nonlinear models produced by  Algorithm~\ref{alg:nonlinear} for the case $\mathcal{D}= \mathcal{D}_{\mathrm{I}}$ with $N=2$ and $3$.}
    \label{linear1D}
\end{figure}

\begin{figure}[!htbp]
    \centering
    \includegraphics[width=.6\textwidth,height=.45\textwidth]{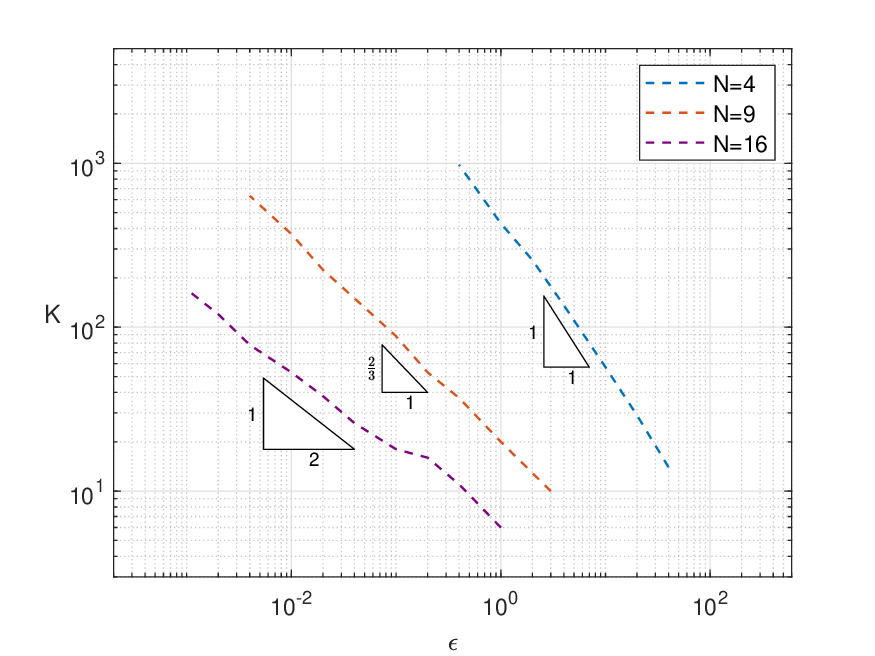}
    \caption{Number of subdomains generated by Algorithm~\ref{alg:nonlinear} to achieve a target tolerance $\epsilon$ for the case $\mathcal{D}= \mathcal{D}_{\mathrm{II}}$ with $N=4,9,$ and $16$ .}
    \label{comp2D}
\end{figure}

Now we consider the two-parameter case, $\mathcal{D} = \mathcal{D}_{\mathrm{II}}$. We run Algorithm~\ref{alg:nonlinear} with an initial parameter $(0,0)$, a size $500$ of a random training set and $\zeta=5$ for the optimality constant $C_{\mathrm{opt}}$. We consider three cases $N=4,9$ and $16$. The number of subdomains $K$ against the tolerance $\epsilon$ is shown in Figure~\ref{comp2D}. Once again, the observed convergence rates are consistent with Theorem~\ref{th conv}.

\begin{figure}[!htbp]
    \centering
    \includegraphics[width=.6\textwidth,height=.5\textwidth]{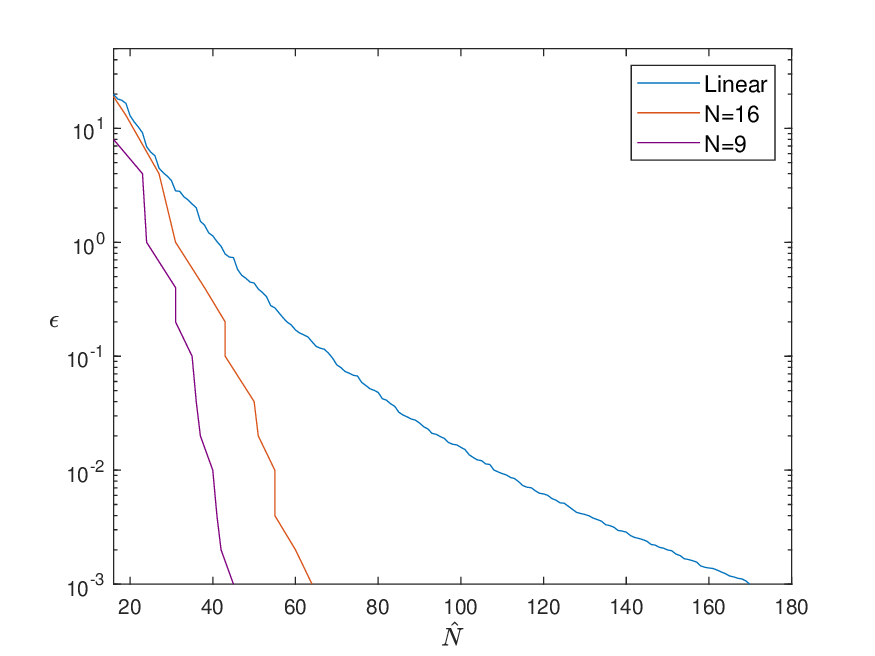}
    \caption{Convergence results for the linear model compared to the nonlinear models produced by  Algorithm~\ref{alg:nonlinear} for the case $\mathcal{D}= \mathcal{D}_{\mathrm{II}}$ with  $N=9$ and $N=16$.}
    \label{linear2D}
\end{figure}

For the online cost comparison,  we build a linear reduced model by executing Algorithm~\ref{alg:PODGreedy} under the configuration $N = \infty$, $\mu^{(1)} = (0,0)$, $\epsilon = 10^{-3}$, and a training set consisting of $5 \times 10^3$ samples. During each iteration, only a single POD mode is incorporated into the reduced basis. With this setup, the prescribed tolerance is met when the dimension of the reduced basis reaches $170$. The convergence behavior of the linear model, together with the nonlinear models produced by Algorithm~\ref{alg:nonlinear} for $N = 9$ and $N = 16$, is illustrated in Figure~\ref{linear2D}. As observed, the nonlinear approach yields a significant reduction in the online computational cost. In particular, for the case $N = 9$, this improvement is achieved with a reduced basis space whose dimension is approximately one quarter of that required by the linear model.

Finally, we examine the role of the optimality constant $C_{\mathrm{opt}}$. As discussed in Section~\ref{pract}, careful selection of $C_{\mathrm{opt}}$ is essential for balancing offline and online performance. To illustrate this, we construct two nonlinear models for $\mathcal{D}_{\mathrm{II}}$ with $N = 16$, initial parameter $(0,0)$, and a training set of size $500$. We consider two choices of $C_{\mathrm{opt}}$, corresponding to $\zeta = 5$ and $\zeta = 50$.

Figure~\ref{opt_comp}(a) displays the number of subdomains required to achieve a given tolerance for both cases. The expected convergence is achieved for both choices. It is evident that the choice $\zeta = 50$ leads to a higher offline computational cost as it necessitates a finer partition of the parameter domain due to the smaller number of POD modes retained at each iteration of the algorithm. Although this finer partition reduces $\hat{\ell}_k$ in \eqref{POD criterion}, thereby requiring more POD modes to satisfy the POD criterion, it also reduces the error introduced by the greedy reduction of the solution space on each subdomain. Since the POD is applied to the error trajectory rather than directly to the solution, these two effects compensate for each other, allowing the prescribed tolerance to be achieved while maintaining a smaller maximum RB-space dimension $\hat N$. This ultimately results in an improved online stage, as shown in Figure~\ref{opt_comp}(b).

Furthermore, we compare two nonlinear models constructed over $\mathcal{D}_{\mathrm{II}}$, both initialized at $(0,0)$, with a training set of size $500$ and tolerance $\epsilon = 10^{-3}$. The first model is built with $N = 9$ and $\zeta = 5$, while the second uses $N = 16$ and $\zeta = 50$. We observe that the online computational costs are comparable, as reflected by the maximum reduced basis dimensions of $50$ and $45$, respectively. However, the second configuration yields a substantially more efficient offline stage and significantly reduced storage requirements, with a library size of $466$ compared to $1641$ for the first model. This improvement is due to the larger value of $N=16$, which enhances the convergence rate according to Theorem~\ref{th conv}, thereby reducing the number of subdomains required to achieve the prescribed tolerance. Moreover, the larger value of $\zeta$ reduces the number of POD modes retained at each iteration of Algorithm~\ref{alg:nonlinear}, which contributes to the smaller values of $\hat{N}$ and thus improves the online efficiency. This clearly highlights the critical role of the optimality constant $C_{\mathrm{opt}}$ in balancing offline and online efficiency. Therefore, we recommend using reasonably large values of both $N$ and $\zeta$ to achieve an effective balance between offline and online computational costs while substantially reducing storage requirements.

\begin{figure}[!htbp]
    \centering
    \subfloat[]{
    \includegraphics[width=.48\textwidth]{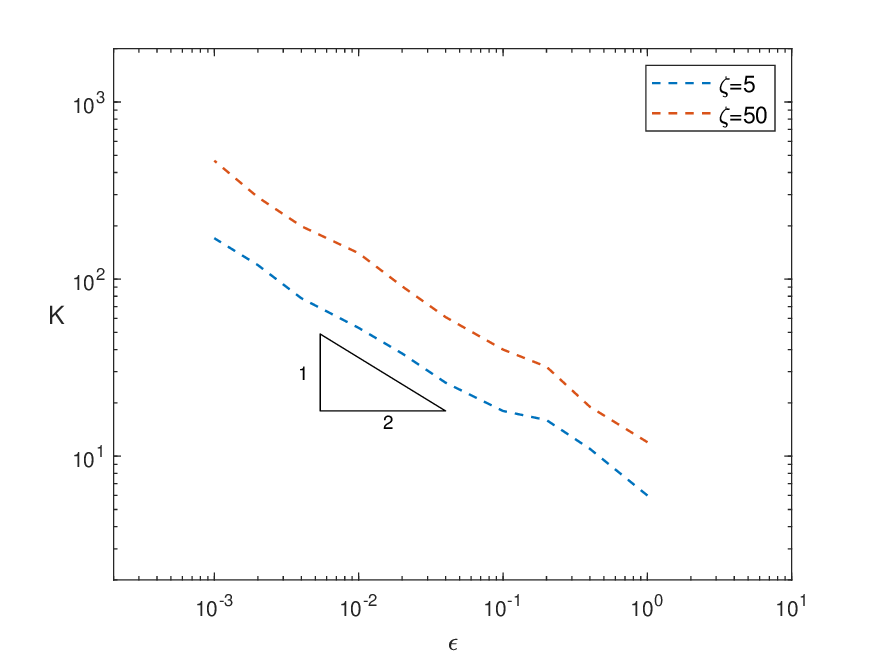}}
    \subfloat[]{
    \includegraphics[width=.48\textwidth]{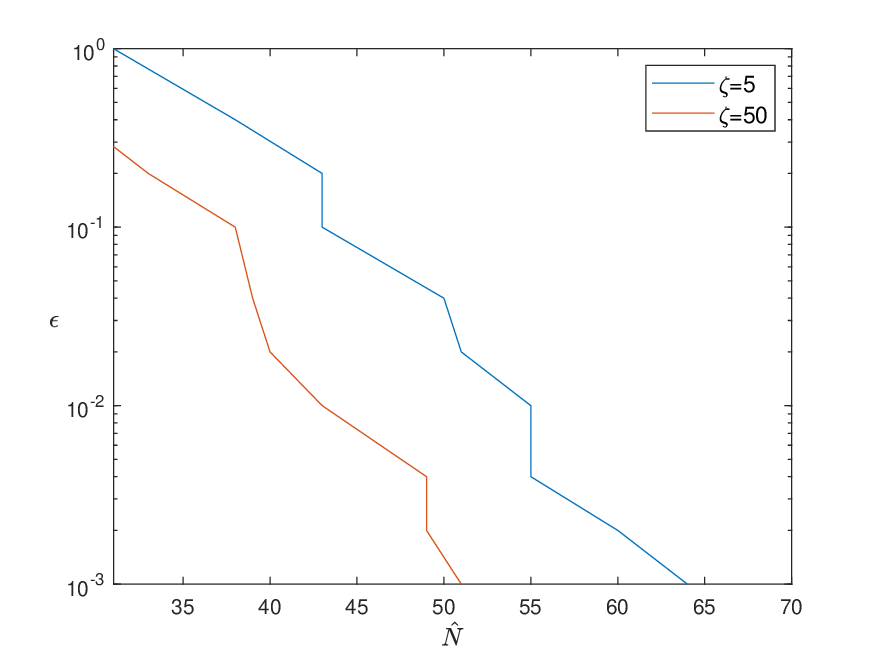}}
   \caption{Comparison of offline and online computational costs for the nonlinear model across different values of $C_{\mathrm{opt}}$, evaluated at $\zeta = 5$ and $\zeta = 50$.}
    \label{opt_comp}
\end{figure}

\bibliographystyle{acm}
\bibliography{ref}

@article{schwab2009space,
  title={Space-time adaptive wavelet methods for parabolic evolution problems},
  author={Schwab, Christoph and Stevenson, Rob},
  journal={Mathematics of Computation},
  volume={78},
  number={267},
  pages={1293--1318},
  year={2009}
}

@article{barrault2004empirical,
  title={An ‘empirical interpolation’ method: application to efficient reduced-basis discretization of partial differential equations},
  author={Barrault, Maxime and Maday, Yvon and Nguyen, Ngoc Cuong and Patera, Anthony T},
  journal={Comptes Rendus Mathematique},
  volume={339},
  number={9},
  pages={667--672},
  year={2004},
  publisher={Elsevier}
}

@article{grepl2007efficient,
  title={Efficient reduced-basis treatment of nonaffine and nonlinear partial differential equations},
  author={Grepl, Martin A and Maday, Yvon and Nguyen, Ngoc C and Patera, Anthony T},
  journal={ESAIM: Mathematical Modelling and Numerical Analysis},
  volume={41},
  number={3},
  pages={575--605},
  year={2007},
  publisher={EDP Sciences}
}

@article{haasdonk2008reduced,
  title={Reduced basis method for finite volume approximations of parametrized linear evolution equations},
  author={Haasdonk, Bernard and Ohlberger, Mario},
  journal={ESAIM: Mathematical Modelling and Numerical Analysis},
  volume={42},
  number={2},
  pages={277--302},
  year={2008},
  publisher={EDP Sciences}
}

@article{grepl2005posteriori,
  title={A posteriori error bounds for reduced-basis approximations of parametrized parabolic partial differential equations},
  author={Grepl, Martin A and Patera, Anthony T},
  journal={ESAIM: Mod{\'e}lisation math{\'e}matique et analyse num{\'e}rique},
  volume={39},
  number={1},
  pages={157--181},
  year={2005}
}

@book{hesthaven2016certified,
  title={Certified reduced basis methods for parametrized partial differential equations},
  author={Hesthaven, Jan S and Rozza, Gianluigi and Stamm, Benjamin and others},
  volume={590},
  year={2016},
  publisher={Springer}
}

@article{binev2011convergence,
  title={Convergence rates for greedy algorithms in reduced basis methods},
  author={Binev, Peter and Cohen, Albert and Dahmen, Wolfgang and DeVore, Ronald and Petrova, Guergana and Wojtaszczyk, Przemyslaw},
  journal={SIAM journal on mathematical analysis},
  volume={43},
  number={3},
  pages={1457--1472},
  year={2011},
  publisher={SIAM}
}

@article{buffa2012priori,
  title={A priori convergence of the Greedy algorithmfor the parametrized reduced basis method},
  author={Buffa, Annalisa and Maday, Yvon and Patera, Anthony T and Prud’Homme, Christophe and Turinici, Gabriel},
  journal={ESAIM: Mathematical modelling and numerical analysis},
  volume={46},
  number={3},
  pages={595--603},
  year={2012},
  publisher={EDP Sciences}
}

@article{knezevic2011reduced,
  title={Reduced basis approximation and a posteriori error estimation for the parametrized unsteady Boussinesq equations},
  author={Knezevic, David J and Nguyen, Ngoc-Cuong and Patera, Anthony T},
  journal={Mathematical Models and Methods in Applied Sciences},
  volume={21},
  number={07},
  pages={1415--1442},
  year={2011},
  publisher={World Scientific}
}

@article{nguyen2009reduced,
  title={Reduced basis approximation and a posteriori error estimation for the time-dependent viscous Burgers’ equation},
  author={Nguyen, Ngoc-Cuong and Rozza, Gianluigi and Patera, Anthony T},
  journal={Calcolo},
  volume={46},
  number={3},
  pages={157--185},
  year={2009},
  publisher={Springer}
}

@article{huynh2007successive,
  title={A successive constraint linear optimization method for lower bounds of parametric coercivity and inf--sup stability constants},
  author={Huynh, Dinh Bao Phuong and Rozza, Gianluigi and Sen, Sugata and Patera, Anthony T},
  journal={Comptes Rendus. Math{\'e}matique},
  volume={345},
  number={8},
  pages={473--478},
  year={2007}
}

@article{rozza2008reduced,
  title={Reduced basis approximation and a posteriori error estimation for affinely parametrized elliptic coercive partial differential equations: application to transport and continuum mechanics},
  author={Rozza, Gianluigi and Huynh, Dinh Bao Phuong and Patera, Anthony T},
  journal={Archives of computational methods in engineering},
  volume={15},
  number={3},
  pages={229--275},
  year={2008},
  publisher={Springer}
}

@article{barakat2026convergence,
  title = {Convergence analysis for a tree-based nonlinear reduced basis method},
  author = {Barakat, Mohamed and Guignard, Diane},
  journal = {IMA Journal of Numerical Analysis},
  pages = {drag035},
  year = {2026},
  publisher = {Oxford University Press},
}

@article{almroth1978automatic,
  title={Automatic choice of global shape functions in structural analysis},
  author={Almroth, Bo O and Stern, Perry and Brogan, Frank A},
  journal={Aiaa Journal},
  volume={16},
  number={5},
  pages={525--528},
  year={1978}
}

@article{noor1980reduced,
  title={Reduced basis technique for nonlinear analysis of structures},
  author={Noor, Ahmed K and Peters, Jeanne M},
  journal={Aiaa journal},
  volume={18},
  number={4},
  pages={455--462},
  year={1980}
}

@article{porsching1985estimation,
  title={Estimation of the error in the reduced basis method solution of nonlinear equations},
  author={Porsching, TA804937},
  journal={Mathematics of Computation},
  volume={45},
  number={172},
  pages={487--496},
  year={1985}
}

@article{rheinboldt1992theory,
  title={On the theory and error estimation of the reduced basis method for multi-paramete problems},
  author={Rheinboldt, Werner C},
  journal={Nonlinear Analysis: Theory, Methods and Applications},
  volume={21},
  pages={849--858},
  year={1993}
}

@article{boyaval2008reduced,
  title={Reduced-basis approach for homogenization beyond the periodic setting},
  author={Boyaval, S{\'e}bastien},
  journal={Multiscale Modeling \& Simulation},
  volume={7},
  number={1},
  pages={466--494},
  year={2008},
  publisher={SIAM}
}

@article{boyaval2009reduced,
  title={A reduced basis approach for variational problems with stochastic parameters: Application to heat conduction with variable Robin coefficient},
  author={Boyaval, S{\'e}bastien and Le Bris, Claude and Maday, Yvon and Nguyen, Ngoc Cuong and Patera, Anthony T},
  journal={Computer Methods in Applied Mechanics and Engineering},
  volume={198},
  number={41-44},
  pages={3187--3206},
  year={2009},
  publisher={Elsevier}
}

@article{nguyen2010reduced,
  title={Reduced basis approximation and a posteriori error estimation for parametrized parabolic {PDE}s: application to real-time Bayesian parameter estimation},
  author={Nguyen, NC and Rozza, Gianluigi and Huynh, DB Phuong and Patera, Anthony T},
  journal={Large-Scale Inverse Problems and Quantification of Uncertainty},
  pages={151--177},
  year={2010},
  publisher={Wiley Online Library}
}

@article{amsallem2012nonlinear,
  title={Nonlinear model order reduction based on local reduced-order bases},
  author={Amsallem, David and Zahr, Matthew J and Farhat, Charbel},
  journal={International Journal for Numerical Methods in Engineering},
  volume={92},
  number={10},
  pages={891--916},
  year={2012},
  publisher={Wiley Online Library}
}

@inproceedings{washabaugh2012nonlinear,
  title={Nonlinear model reduction for CFD problems using local reduced-order bases},
  author={Washabaugh, Kyle and Amsallem, David and Zahr, Matthew and Farhat, Charbel},
  booktitle={42nd AIAA Fluid Dynamics Conference and Exhibit},
  pages={2686},
  year={2012}
}

@article{amsallem2015fast,
  title={Fast local reduced basis updates for the efficient reduction of nonlinear systems with hyper-reduction},
  author={Amsallem, David and Zahr, Matthew J and Washabaugh, Kyle},
  journal={Advances in Computational Mathematics},
  volume={41},
  number={5},
  pages={1187--1230},
  year={2015},
  publisher={Springer}
}

@article{peherstorfer2014localized,
  title={Localized discrete empirical interpolation method},
  author={Peherstorfer, Benjamin and Butnaru, Daniel and Willcox, Karen and Bungartz, Hans-Joachim},
  journal={SIAM Journal on Scientific Computing},
  volume={36},
  number={1},
  pages={A168--A192},
  year={2014},
  publisher={SIAM}
}

@article{amsallem2016pebl,
  title={{PEBL-ROM}: Projection-error based local reduced-order models},
  author={Amsallem, David and Haasdonk, Bernard},
  journal={Advanced Modeling and Simulation in Engineering Sciences},
  volume={3},
  number={1},
  pages={6},
  year={2016},
  publisher={Springer}
}

@inproceedings{haasdonk2008adaptive,
  title={Adaptive basis enrichment for the reduced basis method applied to finite volume schemes},
  author={Haasdonk, Bernard and Ohlberger, Mario},
  booktitle={Proc. 5th International Symposium on Finite Volumes for Complex Applications},
  pages={471--478},
  year={2008}
}

@article{haasdonk2011training,
  title={A training set and multiple bases generation approach for parameterized model reduction based on adaptive grids in parameter space},
  author={Haasdonk, Bernard and Dihlmann, Markus and Ohlberger, Mario},
  journal={Mathematical and Computer Modelling of Dynamical Systems},
  volume={17},
  number={4},
  pages={423--442},
  year={2011},
  publisher={Taylor \& Francis}
}

@article{eftang2010hp,
  title={An ``hp''  certified reduced basis method for parametrized elliptic partial differential equations},
  author={Eftang, Jens L and Patera, Anthony T and R{\o}nquist, Einar M},
  journal={SIAM Journal on Scientific Computing},
  volume={32},
  number={6},
  pages={3170--3200},
  year={2010},
  publisher={SIAM}
}

@incollection{guignard2024tree,
  title={Tree-based nonlinear reduced modeling},
  author={Guignard, Diane and Mula, Olga},
  booktitle={Multiscale, Nonlinear and Adaptive Approximation II},
  pages={267--298},
  year={2024},
  publisher={Springer}
}

@article{ohlberger2013reduced,
  title={Reduced basis methods: Success, limitations and future challenges},
  author={Ohlberger, Mario and Rave, Stephan},
  journal={Proceedings of the Conference Algoritmy},
  pages={1--12},
  year={2013}
}

@article{ohlberger2016nonlinear,
  title={Nonlinear Approximation Methods for High-Dimensional Parametric PDEs},
  author={Ohlberger, Mario and Rave, Stephan},
  journal={Model Reduction of Parametrized Systems},
  pages={95--116},
  year={2017},
  publisher={Springer}
}

@article{greif2019tensor,
  title={A Posteriori Error Bounds for the Reduced Basis Method for Transport Dominated Problems},
  author={Greif, Constantin and Urban, Karsten},
  journal={ESAIM: Mathematical Modelling and Numerical Analysis},
  volume={53},
  number={6},
  pages={1965--1994},
  year={2019}
}

@article{peherstorfer2022nonlinear,
  title={Model Reduction for Transport-Dominated Problems},
  author={Peherstorfer, Benjamin},
  journal={Acta Numerica},
  volume={31},
  pages={507--595},
  year={2022}
}

@article{quarteroni2011certified,
  title={Certified reduced basis approximation for parametrized partial differential equations and applications},
  author={Quarteroni, Alfio and Rozza, Gianluigi and Manzoni, Andrea},
  journal={Journal of Mathematics in Industry},
  volume={1},
  number={1},
  pages={3},
  year={2011},
  publisher={Springer}
}

@article{dihlmann2011model,
  title={Model reduction of parametrized evolution problems using the reduced basis method with adaptive time-partitioning},
  author={Dihlmann, Markus and Drohmann, Martin and Haasdonk, Bernard},
  journal={Proc. of ADMOS},
  volume={2011},
  pages={64},
  year={2011}
}

@inproceedings{drohmann2011adaptive,
  title={Adaptive reduced basis methods for nonlinear convection--diffusion equations},
  author={Drohmann, Martin and Haasdonk, Bernard and Ohlberger, Mario},
  booktitle={Finite Volumes for Complex Applications VI Problems \& Perspectives: FVCA 6, International Symposium, Prague, June 6-10, 2011},
  pages={369--377},
  year={2011},
  organization={Springer}
}

@article{copeland2022reduced,
  title={Reduced order models for Lagrangian hydrodynamics},
  author={Copeland, Dylan Matthew and Cheung, Siu Wun and Huynh, Kevin and Choi, Youngsoo},
  journal={Computer Methods in Applied Mechanics and Engineering},
  volume={388},
  pages={114259},
  year={2022},
  publisher={Elsevier}
}

@article{shimizu2021windowed,
  title={Windowed space--time least-squares Petrov--Galerkin model order reduction for nonlinear dynamical systems},
  author={Shimizu, Yukiko S and Parish, Eric J},
  journal={Computer Methods in Applied Mechanics and Engineering},
  volume={386},
  pages={114050},
  year={2021},
  publisher={Elsevier}
}

@article{hesthaven2026nonlinear,
  title={Nonlinear model reduction for transport-dominated problems},
  author={Hesthaven, Jan S and Peherstorfer, Benjamin and Unger, Benjamin},
  journal={arXiv preprint arXiv:2602.01397},
  year={2026}
}

@article{eftang2011hp2,
  title={An hp certified reduced basis method for parametrized parabolic partial differential equations},
  author={Eftang, Jens L and Knezevic, David J and Patera, Anthony T},
  journal={Mathematical and Computer Modelling of Dynamical Systems},
  volume={17},
  number={4},
  pages={395--422},
  year={2011},
  publisher={Taylor \& Francis}
}

@article{maday2002priori,
  title={A priori convergence theory for reduced-basis approximations of single-parameter elliptic partial differential equations},
  author={Maday, Yvon and Patera, Anthony T and Turinici, Gabriel},
  journal={Journal of Scientific Computing},
  volume={17},
  number={1},
  pages={437--446},
  year={2002},
  publisher={Springer}
}
\end{document}